\documentclass[12pt, twoside, leqno]{article}

\usepackage{amsmath,amsthm}
\usepackage{amssymb}
\usepackage{hyperref}
\usepackage{enumerate}

\usepackage{graphicx}

\newtheorem{thm}{Theorem}[section]
\newtheorem{cor}[thm]{Corollary}
\newtheorem{lem}[thm]{Lemma}

\theoremstyle{definition}

\numberwithin{equation}{section}

\def\op{{\mathtt{op}}}
\def\rSU2{{\rm SU(2)}}
\def\TT{{\mathbb T}}
\def\NN{{\mathbb N}}
\def\HS{{\mathtt{HS}}}
\def\ZZ{{\mathbb Z}}
\def\SU2{{\widehat{\rSU2}}}
\def\Gh{{\widehat{G}}}

\DeclareMathOperator{\Tr}{Tr}
\DeclareMathOperator{\diag}{diag}

\begin{document}


\baselineskip=17pt


\title{Hardy-Littlewood-Paley inequalities\\ and Fourier multipliers on SU(2)}

\author{Rauan Akylzhanov
\\
Department of Mathematics
\\
Imperial College London
\\
180 Queen's Gate, London SW7 2AZ
\\
United Kingdom
\\
{\it E-mail address} {\rm r.akylzhanov14@imperial.ac.uk}
\and
Erlan Nurlustanov
\\
Department of Mathematics
\\
Moscow State University, Kazakh Branch
\\
Astana, Kazakhstan
\\
{\it E-mail address} {\rm er-nurs@yandex.ru}
\and
Michael Ruzhansky
\\
Department of Mathematics
\\
Imperial College London
\\
180 Queen's Gate, London SW7 2AZ
\\
United Kingdom
\\
{\it E-mail address} {\rm m.ruzhansky@imperial.ac.uk}
}
\maketitle

\renewcommand{\thefootnote}{}
\footnote{The third
 author was supported by the EPSRC Grant EP/K039407/1.}

\renewcommand{\thefootnote}{\arabic{footnote}}
\setcounter{footnote}{0}


\begin{abstract}
In this paper we prove noncommutative versions of Hardy--Little\-wood 
and Paley inequalities
relating a function and its
Fourier coefficients on the group $\rSU2$.
As a consequence, we use it to obtain 
lower bounds for the $L^p$--$L^q$ norms of Fourier  multipliers on 
the group $\rSU2$, for $1<p\leq 2\leq q<\infty$. 
In addition, we give upper bounds of a similar form, analogous to the known
results on the torus, but now in the noncommutative setting of $\rSU2$.
\end{abstract}
\maketitle


\section{Introduction}

Let $\mathbb{T}^n$ be the $n$-dimensional torus and let $1<p\leq q<\infty$. 
A sequence $\lambda=\{\lambda_k\}_{k\in\mathbb{Z}^n}$ of complex numbers is 
said to be a multiplier of trigonometric Fourier series from 
$L^p(\mathbb{T}^n)$ to $L^q(\mathbb{T}^n)$ if 
the operator 
$$
T_{\lambda} f(x) = \sum_{k\in\mathbb{Z}^n}\lambda_k\widehat{f}(k)e^{ i kx}
$$
is bounded from $L^p(\mathbb{T}^n)$ to $L^q(\mathbb{T}^n)$. 
We denote by ${\bf m}^q_p$ the set of such multipliers.

Many problems in harmonic analysis and partial differential equations can be reduced to the boundedness of multiplier transformations. There arises a natural question of finding sufficient conditions for $\lambda\in{\bf m}^p_p$. 
The topic of ${\bf m}^q_p$ multipliers has been extensively researched. 
Using methods such as the Littlewood-Paley decomposition and Calderon-Zygmund theory,
it is possible to prove H\"ormander-Mihlin type theorems, see e.g.
Mihlin \cite{Mihlin:Lp-Doklady,Mihlin:Lp-LGU},
H\"ormander \cite{Hormander:invariant-LP-Acta-1960}, and later works.
 
Multipliers have been then analysed in a variety of different settings, see e.g.
Gaudry \cite{Gaudry:LpLq-loc-cpt-abel-gps-PJM-1966}, 
Cowling \cite{Cowling:PhD},
Vretare \cite{Vretare:MS-1974}.
The literature on the spectral multipliers is too rich to be reviewed here,
see e.g. a recent paper 
\cite{Cowling-Klima-Sikora:spectral-TAMS-2011}
and references therein. The same is true for multipliers on locally
compact abelian groups, see e.g. 
\cite{Arhancet:CM-2012}, or for Fourier or spectral multipliers on symmetric spaces,
see e.g. \cite{Anker:AM-1990} or \cite{Cowling-Giulini-Meda:DMJ-1993}, resp.
We refer to the above and to other papers for further references on the history of
${\bf m}^q_p$ multipliers on spaces of different types.

In this paper we are interested in questions for Fourier multipliers on compact Lie
groups, in which case the literature is much more sparse: in the sequel we will make
a more detailed review of the existing results. Thus, in this paper 
we will be investigating several questions
in the model case of Fourier multipliers on the compact group $\rSU2$.
Although we will not explore it in this paper,
we note that there are links between multipliers on $\rSU2$ and those
on the Heisenberg group, see Ricci and Rubin
\cite{Ricci-Rubin:transferring-F-mult-SU2-AJM-1986}.

In general, most of the multiplier theorems imply that 
$\lambda\in {\bf m}^p_p$ for all $1<p<\infty$ at once. 
In \cite{Stein1970}, Stein raised the question of finding more subtle sufficient conditions for a multiplier to belong 
to some ${\bf m}^p_p$, $p\neq 2$, without implying also that it belongs to all ${\bf m}^p_p$, $1<p<\infty$. 
In \cite{Nurs_Tleukh_Mult}, Nursultanov and Tleukhanova provided conditions on 
$\lambda=\{\lambda_k\}_{k\in\mathbb{Z}}$ to belong to ${\bf m}^q_p$ for the range
$1<p\leq 2\leq q <\infty$. In particular, they established lower and upper bounds for the 
norms of multiplier $\lambda\in{\bf m}^q_p$ which depend on parameters $p$ and $q$. 
Thus, this provided a partial answer to Stein's question.  Let us recall this result in the case $n=1$:

\begin{thm}\label{THM:NT}
Let $1<p\leq 2 \leq q <\infty$ and let 
$M_0$ denote the set of all finite arithmetic sequences in $\mathbb{Z}$. Then the following inequalities hold:
$$
	\sup_{Q\in M_0}\frac1{|Q|^{1+\frac1q-\frac1p}}\left|\sum_{m\in Q}\lambda_m\right|
\lesssim
\|T_\lambda\|_{L^p\to L^q}
\lesssim
\sup_{k\in\mathbb{N}}\frac1{k^{1+\frac1q-\frac1p}}\sum^k_{m=1}\lambda^{*}_m,
$$
where $\lambda^*_m$ is a non-increasing rearrangement of $\lambda_m$, and $|Q|$ is the number of elements in the 
arithmetic progression $Q$ .
\end{thm}

In this paper we study the noncommutative versions of this and other related results.
As a model case, we concentrate on analysing
Fourier multipliers between Lebesgue spaces on the group $\rSU2$ of $2\times 2$ unitary
matrices with determinant one. 
Sufficient conditions for Fourier multipliers on $\rSU2$ to be bounded on $L^p$-spaces have been analysed
by Coifman-Weiss 
\cite{Coifman-Weiss:SU2-Argentina-1970} and
Coifman-de Guzman \cite{Coifman-deGuzman:SU2-Argentina-1970}, see also Chapter 5 in
Coifman and Weiss' book \cite{coifman+weiss_lnm}, and are given in terms of the Clebsch-Gordan coefficients
of representations on the group $\rSU2$. A more general perspective was provided in
\cite{RuWi2013} where conditions on Fourier multipliers to be bounded on $L^p$ were obtained for 
general compact Lie groups, and Mihlin-H\"ormander theorems on general compact Lie groups have been established in \cite{RuWi2015}.

Results about spectral multipliers are more known, for functions of the Laplacian
(N. Weiss \cite{Weiss_N:Lp-biinvariant} or 
Coifman and Weiss \cite{Coifman-Weiss:central-multipliers-BAMS}), 
or of the sub-Laplacian on $\rSU2$,
see Cowling and Sikora \cite{Cowling-Sikora:SU2-sublaplacian}.
However, following \cite{Coifman-Weiss:SU2-Argentina-1970, coifman+weiss_lnm, RuWi2013, RuWi2015},
here were are rather interested in Fourier multipliers.

In this paper we obtain lower and upper estimates for the norms of Fourier 
multipliers acting between $L^p$ 
and $L^q$ spaces on $\rSU2$. These estimates explicitly depend on parameters $p$ and $q$.   
Thus, this paper can be regarded as a contribution to Stein's question in the 
noncommutative setting of $\rSU2$. 
At the same time we provide a noncommutative analogue of 
Theorem \ref{THM:NT}.
Briefly, let $A$ be the Fourier multiplier on $\rSU2$ given by 
$$\widehat{Af}(l)=\sigma_A(l)\widehat{f}(l), \textrm{ for }
\sigma_A(l)\in\mathbb C^{(2l+1)\times (2l+1)},\;
l\in\frac12\mathbb{N}_0,$$
where we refer to Section \ref{SEC:Hardy_Littlewood_Paley_inequalities} for definitions and notation
related to the Fourier analysis on $\rSU2$. For such operators, in
Theorem \ref{THM:lower-bound}, for
$1<p\leq 2 \leq q <\infty$, we give two lower bounds, one of which is of the form
\begin{equation}
\label{SU_conds-3}
\sup_{l\in\frac12\mathbb{N}_0}
\frac{1}{(2l+1)^{1+\frac1{q}-\frac1{p}}}
\frac{1}{2l+1}
\left|\Tr \sigma_A(l)\right|
\lesssim
\|A\|_{L^p(\rSU2)\to L^q(\rSU2)}.
\end{equation}
A related upper bound 
\begin{equation}
\label{transition3}
	\|A\|_{L^p(\rSU2)\to L^q(\rSU2)} 
\lesssim
\sup_{s>0}
s
\left(
\sum\limits_{\substack{l\in\frac12\mathbb{N}_0\\ \|\sigma_A(l)\|_{op}\geq s}}
(2l+1)^2
\right)^{\frac1p-\frac1q}.
\end{equation}
will be given in Theorem \ref{THM:upper}.

The proof of the lower bound
is based on the new inequalities describing the relationship between  
the ``size'' of a function and the 
``size'' of its Fourier transform. These inequalities can be viewed as a noncommutative 
$\rSU2$-version of the
Hardy-Littlewood inequalities obtained by Hardy and Littlewood in \cite{HL}.
To explain this briefly, we recall that in \cite{HL}, Hardy and Littlewood have shown
that for $1<p\leq 2$ and $f\in L^p(\TT)$, the following inequality holds true:
\begin{equation}
\label{H_L_inequality-0}
	\sum_{m \in \ZZ}{(1+|m|)}^{p-2}|\widehat{f}(m)|^{p} \leq K\|f\|^p_{L^p(\TT)},
\end{equation}
arguing this to be a suitable extension of the Plancherel identity to $L^p$-spaces.
While we refer to Section \ref{necessary_background} and to
Theorem \ref{THM:Hardy_Littlewood} for more details on this, our analogue for this is
the inequality 
\begin{equation}
\label{sufficient_SU-2-0}
	\sum_{l\in\frac12\mathbb{N}_0} (2l+1)
	(2l+1)^{\frac52(p-2)}\|\widehat{f}(l)\|^p_{\HS}
\leq c \|f\|^p_{L^p(\rSU2)},\quad 1<p\leq 2,
\end{equation}
which for $p=2$ gives the ordinary Plancherel identity on $\rSU2$,
see \eqref{EQ:plancherel}. We refer to
Theorem \ref{Akylzhanov_2} for this and to Corollary 
\ref{COR:duality}  for the dual statement.
For $p\geq 2$, the necessary conditions for a function to
belong to $L^p$ are usually harder to obtain.
In Theorem \ref{Akylzhanov_1} we give such a result for $2\leq p <\infty$ which takes the form
\begin{equation}
\label{necess_SU-0}
	\sum_{l\in\frac12\mathbb{N}_0}(2l+1)^{p-2}\left(\sup_{\substack{k\in\frac12\mathbb{N}_0 \\ k\geq l}}\frac1{2k+1}\left|\Tr\widehat{f}(k)\right|\right)^p
\leq 
c\|f\|^p_{L^p(\rSU2)}, \quad 2\leq p <\infty.
\end{equation}
In turn, this gives a noncommutative analogue to the known similar result on the circle
(which we recall in Theorem \ref{NED}). Similar to \eqref{SU_conds-3}, the averaged trace appears 
also in \eqref{necess_SU-0} -- it is the usual trace divided by the number of diagonal
elements in the matrix.

In \cite{Hormander:invariant-LP-Acta-1960} Ho\"rmander proved a Paley-type inequality for 
the Fourier transform on $\mathbb{R}^N$. 
In this paper we obtain an analogue of this inequality on the group $\rSU2$.

The results on the group $\rSU2$ are usually quite important since, in view of the 
resolved Poincar\'e conjecture,
they provide information about corresponding transformations on general 
closed simply-connected three-dimensional manifolds   
(see \cite{RT} for a more detailed outline of such relations).
In our context, they give explicit versions of known results on the circle $\TT$ or
on the torus $\TT^{n}$, in the simplest noncommutative setting of $\rSU2$.

At the same time, we note that some results of this paper can be extended to 
Fourier multipliers on general compact Lie groups. 
However, such analysis requires a more abstract approach, 
and will appear elsewhere.
 
The paper is organised as follows. In Section \ref{SEC:Hardy_Littlewood_Paley_inequalities} 
we fix the notation for 
the representation theory of $\rSU2$ and formulate estimates relating functions with its Fourier
coefficients: the $\rSU2$-version of the Hardy--Littlewood and Paley inequalities and further
extensions. In Section \ref{main_results} we formulate and prove the lower bounds for operator
norms of Fourier multipliers, and in Section
\ref{SEC:upper-bounds} the upper bounds.
Our proofs are based on inequalities from Section \ref{techniques}. In Section \ref{SEC:Hardy_Littlewood_Paley_inequalities_Proofs} we complete the proofs of the results presented in previous sections.
 
We shall use the symbol $C$ to denote various positive constants, and $C_{p,q}$ for constants
which may depend only on indices $p$ and $q$.
We shall write $x\lesssim y$  for the relation $|x|\leq C |y|$, and write $x\cong y$ if $x\lesssim y$ and $y\lesssim x$.

The authors would like to thank V\'eronique Fischer for useful remarks.

\label{necessary_background}

\section{Hardy-Littlewood and Paley inequalities on \texorpdfstring{$\rSU2$}{SU(2)}}
\label{SEC:Hardy_Littlewood_Paley_inequalities}
The aim of this section is to discuss necessary conditions and sufficient conditions for the $L^p(\rSU2)$-integrability 
of a function by means of its Fourier coefficients.
The main results of this section are Theorems \ref{Akylzhanov_2},
\ref{THM:Paley_inequality}  and \ref{Akylzhanov_1}.
These results will provide a noncommutative version of known results of this
type on the circle $\TT$.
The proofs of most of the results of this Section are given in Section \ref{SEC:Hardy_Littlewood_Paley_inequalities_Proofs}.

First, let us fix the notation concerning the representations of the compact Lie group
$\rSU2$. 
There are different types of notation in the literature for the appearing objects - we will follow
the notation of Vilenkin \cite{Vilenkin:BK-eng-1968}, as well as that in
\cite{RT,Ruzhansky+Turunen-IMRN}.
Let us identify $z=(z_1,z_2)\in\mathbb{C}^{1\times2}$, and let $\mathbb{C}[z_1,z_2]$ 
be the space of two-variable polynomials $f\colon \mathbb{C}^2\to\mathbb{C}$. Consider mappings
$$
	t^l\colon \rSU2 \to GL(V_l),\quad (t^l(u)f)(z)=f(zu),
$$
where $l\in\frac12\mathbb{N}_0$ is called the quantum number, 
$\mathbb{N}_0=\mathbb{N}\cup\{0\}$,
and where $V_l$ is the $(2l+1)$-dimensional subspace of $\mathbb{C}[z_1,z_2]$ containing the homogeneous polynomials of order $2l\in\mathbb{N}_0$, i.e.
$$
	V_l=\{f\in\mathbb{C}[z_1,z_2]\colon f(z_1,z_2)=\sum^{2l}_{k=0}a_kz^k_1z^{2l-k}_2,\quad 
	\{a_k\}^{2l}_{k=0}\subset \mathbb{C} \}.
$$
The unitary dual of $\rSU2$ is
$$
	\SU2\cong \{t^l\in {\rm Hom}(\rSU2,{\rm U}(2l+1))\colon l\in\frac12\mathbb{N}_0\},
$$
where ${\rm U}(d)\subset \mathbb{C}^{d\times d}$ is the unitary matrix group, and matrix
components $t^l_{mn}\in C^{\infty}(\rSU2)$ can be written as products of exponentials and Legendre-Jacobi functions,
see Vilenkin \cite{Vilenkin:BK-eng-1968}. It is also customary to let the indices $m,n$ to range from
$-l$ to $l$, equi-spaced with step one.
We define the Fourier transform on $\rSU2$ by 
$$
	\widehat{f}(l):=\int\limits_{\rSU2} f(u)t^l(u)^*\,du,
$$
with the inverse Fourier transform (Fourier series) given by 
$$
	f(u)=\sum_{l\in\frac12\mathbb{N}_0}(2l+1)\Tr \widehat{f}(l)t^l(u).
$$
The Peter-Weyl theorem on $\rSU2$ implies, in particular, that this pair of transforms are
inverse to each 
other and that the Plancherel identity
\begin{equation}
\label{EQ:plancherel}
\|f\|^2_{L^2(\rSU2)} = 
\sum_{l\in\frac12\mathbb{N}_0}(2l+1)\|\widehat{f}(l)\|^2_{\HS}=:\|\widehat{f}\|^2_{\ell^2(\rSU2)}
\end{equation}
holds true for all $f\in L^2(\rSU2)$. 
Here $\|\widehat{f}(l)\|^2_{\HS}=\Tr \widehat{f}(l)\widehat{f}(l)^*$ denotes the Hilbert-Schmidt norm of matrices.
For more details on the Fourier transform on $\rSU2$ and on arbitrary compact Lie groups,
and for subsequent Fourier and operator analysis we can refer to \cite{RT}.

\label{techniques}

There are different ways to compare the ``sizes'' of $f$ and $\widehat{f}$. 
Apart from the Plancherel's identity \eqref{EQ:plancherel}, there are other important relations, such as
the Hausdorff-Young or the Riesz-Fischer theorems. However, such estimates usually require the
change of the exponent $p$ in $L^p$-measurements of $f$ and $\widehat{f}$.
Our first results deal with comparing $f$ and $\widehat{f}$ in the same scale of $L^p$-measurements.
Let us remark on the background of this problem. 
In \cite[Theorems 10 and 11]{HL}, 
Hardy and Littlewood proved the following generalisation of the Plancherel's identity. 
\begin{thm}[Hardy--Littlewood \cite{HL}] 
\label{THM:Hardy_Littlewood}
The following holds.
\begin{enumerate}
\item Let $1<p\leq 2$. If $f\in L^p(\TT)$, then
\begin{equation}
\label{H_L_inequality}
	\sum_{m \in \ZZ}{(1+|m|)}^{p-2}|\widehat{f}(m)|^{p} \leq K_p\|f\|^p_{L^p(\TT)},
\end{equation}
where $K_p$ is a constant which depends only on $p$.

\item  Let $2\leq p<\infty$.  
If $\{\widehat{f}(m)\}_{m\in\ZZ}$ is a sequence of complex numbers such that 
\begin{equation}
\label{H_L_condition}
\displaystyle\sum_{m\in\ZZ} (1+|m|)^{p-2}|\widehat{f}(m)|^p<\infty,
\end{equation}
then there is a function $f\in L^p(\TT)$ with Fourier coefficients given by $\widehat{f}(m)$, and 
$$
	\|f\|^p_{L^p(\TT)}\leq K_{p}^{\prime}\sum_{m\in\ZZ} (1+|m|)^{p-2}|\widehat{f}(m)|^p.
$$
\end{enumerate}
\end{thm}
Hewitt and Ross \cite{HR} generalised this theorem to the setting of compact abelian groups.
Now, we give an analogue
of the Hardy--Littlewood Theorem \ref{THM:Hardy_Littlewood} in the noncommutative setting 
of the compact group $\rSU2$.
\begin{thm}\label{Akylzhanov_2}
If $1<p\leq 2$ and $f\in L^p(\rSU2)$, then we have
\begin{equation}
\label{sufficient_SU}
	\sum_{l\in\frac12\mathbb{N}_0}(2l+1)^{\frac52p-4}\|\widehat{f}(l)\|^p_{\HS}
\leq c_p \|f\|^p_{L^p(\rSU2)}.
\end{equation}
\end{thm}
We can write this in the form more resembling the Plancherel identity, namely, as 
\begin{equation}
\label{sufficient_SU-2}
	\sum_{l\in\frac12\mathbb{N}_0} (2l+1)
	(2l+1)^{\frac52(p-2)}\|\widehat{f}(l)\|^p_{\HS}
\leq c_p \|f\|^p_{L^p(\rSU2)},
\end{equation}
providing a resemblance to both \eqref{H_L_inequality} and 
\eqref{EQ:plancherel}.
By duality, we obtain

\begin{cor} \label{COR:duality} 
If $2\leq p< \infty$ and 
$\sum_{l\in\frac12\mathbb{N}_0}(2l+1)^{\frac52p-4}\|\widehat{f}(l)\|^p_{\HS}<\infty$,
then $f\in L^p(\rSU2)$ and we have
\begin{equation}
\label{duality}
\|f\|^p_{L^p(\rSU2)}
\leq c_{p}
	\sum_{l\in\frac12\mathbb{N}_0}(2l+1)^{\frac52p-4}\|\widehat{f}(l)\|^p_{\HS}.
\end{equation}
\end{cor}
For $p=2$, both of these statements reduce to the Plancherel identity
\eqref{EQ:plancherel}.

\medskip
In \cite{Hormander:invariant-LP-Acta-1960} Ho\"rmander proved a Paley-type inequality for 
the Fourier transform on $\mathbb{R}^N$. 
We now give an analogue of this inequality on the group $\rSU2$.

\begin{thm}\label{THM:Paley_inequality} 
Let $1<p\leq 2$. Suppose 
$\{\sigma(l)\}_{l\in\frac12\NN_0}$
 is a sequence of complex matrices 
 $\sigma(l)\in \mathbb C^{(2l+1)\times (2l+1)}$ such that
\begin{equation}
\label{weak_symbol_estimate}
K_{\sigma}:=\sup_{s>0}s\sum\limits_{\substack{l\in\frac12\NN_0\\\|\sigma(l)\|_{\op}\geq s }}(2l+1)^2<\infty.
\end{equation}
Then we have 
\begin{equation}
\label{EQ:Paley_inequality}
\sum\limits_{l\in\frac12\NN_0}(2l+1)^{p(\frac2p-\frac12)} \|\widehat{f}(l)\|^p_{\HS}\|\sigma(l)\|^{2-p}_{\op}
\lesssim
{K_{\sigma}}^{2-p}
\|f\|^p_{L^p(\rSU2)}.
\end{equation}
\end{thm}
It will be useful to recall the spaces $\ell^{p}(\SU2)$ on the discrete
unitary dual $\SU2$. For general compact Lie groups these spaces have been 
introduced and studied in \cite[Section 10.3]{RT}. In the particular case of 
$\rSU2$, for a sequence of complex matrices
 $\sigma(l)\in \mathbb C^{(2l+1)\times (2l+1)}$
they can be defined by the finiteness of the norms
\begin{equation}\label{EQ:lp}
\|\sigma\|_{\ell^{p}(\SU2)}:=
\left(\sum\limits_{l\in\frac12\NN_0}(2l+1)^{p(\frac2{p}-\frac12)}\|\sigma(l)\|^{p}_{\HS}
\right)^{\frac1{p}}, \quad 1 \leq p <\infty,
\end{equation}
and
\begin{equation}\label{EQ:linf}
\|\sigma\|_{\ell^{\infty}(\SU2)}:=
\sup\limits_{l\in\frac12\NN_0}(2l+1)^{-\frac12}\|\sigma(l)\|_{\HS}.
\end{equation}
Among other things, it was shown in \cite[Section 10.3]{RT} that these
spaces are interpolation spaces, they satisfy the duality property and, with
$\sigma=\widehat{f}$,
 the  Hausdorff-Young inequality 
\begin{equation}
\label{H-Y}
\left(
\sum\limits_{l\in\frac12\NN_0}(2l+1)^{p'(\frac2{p'}-\frac12)}\|\widehat{f}(l)\|^{p'}_{\HS}
\right)^{\frac1{p'}}
\equiv
\|\widehat{f}\|_{\ell^{p'}(\SU2)}\lesssim \|f\|_{L^p(\rSU2)},\; 1 \leq p \leq 2.
\end{equation}
Further, we recall a result on the interpolation of weighted spaces from \cite{BL2011}:
\begin{thm}[Interpolation of weighted spaces]
\label{THM:L_p-weighted-interpolation}
 Let us write $d\mu_0(x)=\omega_0(x)d\mu(x)$, $d\mu_1(x)=\omega_1(x)d\mu(x)$, and write $L^p(\omega)=L^p(\omega d\mu)$ for the weight $\omega$. Suppose that $0<p_0,p_1<\infty$. Then 
$$
	(L^{p_0}(\omega_0), L^{p_1}(\omega_1))_{\theta,p}=L^p(\omega),
$$
where $0<\theta<1,\frac1p=\frac{1-\theta}{p_0}+\frac{\theta}{p_1}$, and $\omega=w^{p\frac{1-\theta}{p_0}}_0w^{p\frac{\theta}{p_1}}_1$.
\end{thm}
From this we obtain:
\begin{cor}
\label{Cor:general_Paley_inequality}
Let $1<p\leq b \leq p'<\infty$.
If $\{\sigma(l)\}_{l\in\frac12\NN_0}$ satisfies condition \eqref{weak_symbol_estimate} with constant $K_{\sigma}$, then we have
\begin{multline}
\label{EQ:general_Paley_inequality}
\left(
\sum\limits_{l\in\frac12\NN_0}(2l+1)^{b(\frac2b-\frac12)}
\left(\|\widehat{f}(l)\|_{\HS} 
\|\sigma(l)\|^{\frac1b-\frac{1}{p'}}_{\op}
\right)^b
\right)^{\frac1b}
\\ \lesssim
\left(K_{\sigma}\right)^{\frac1b-\frac1{p'}}
\|f\|_{L^p(\rSU2)}.
\end{multline}
\end{cor}
This reduces to \eqref{H-Y} when $b=p'$ and to \eqref{EQ:Paley_inequality} when $b=p$.
\begin{proof}
We consider a sub-linear operator $A$ which takes a function $f$ to its Fourier transform $\widehat{f}(l)$ divided by $\sqrt{2l+1}$
i.e.
	$$
	f \mapsto Af=:\left\{\frac{\widehat{f}(l)}{\sqrt{2l+1}}\right\}_{l\in\frac12\mathbb{N}_0},
	$$
	where
	$$
	\widehat{f}(l)=\int\limits_{\rSU2}f(u)t^l(u)^*\,u \in \mathbb{C}^{(2l+1)\times (2l+1)},\; l\in\frac12\mathbb{N}_0.
	$$

The statement follows from Theorem \ref{THM:L_p-weighted-interpolation} if we regard 
the left-hand sides of inequalities \eqref{EQ:Paley_inequality} and
\eqref{H-Y} as an $\|Af\|_{L^p}$-norm in a weighted sequence space over $\frac12\NN_0$ with the weights given by $w_0(l)=(2l+1)^2\|\sigma(l)\|^{2-p}_{op}$ and $w_1(l)=(2l+1)^2,\,l\in\frac12\NN_0$.
\end{proof}
\medskip


\medskip
Coming back to the Hardy--Littlewood Theorem \ref{THM:Hardy_Littlewood},
we see that the convergence of the series \eqref{H_L_condition} is a sufficient condition 
for $f$ to belong to $L^p(\TT)$, for $p\geq 2$. However, this condition is not necessary.
Hence, there arises the question of finding necessary conditions for $f$ to belong to 
$L^p$. In other words, there is the problem of finding lower estimates for $\|f\|_{L^p}$ 
in terms of the series of the form \eqref{H_L_condition}.  
Such result on $L^p(\mathbb{T})$ was obtained by Nursultanov and can be stated as follows.
\begin{thm}[\cite{NED}]\label{NED}
If $2<p<\infty$ and $f\in L^p(\mathbb{T})$, then we have
\begin{equation}
\label{necess_T}
	\sum^{\infty}_{k=1}k^{p-2}\left(\sup_{\substack{e\in M \\ |e|\geq k}}
	\frac1{|e|}\left|\sum_{m\in e}\widehat{f}(m)\right|\right)^p\leq C \|f\|^p_{L^p(\TT)},
\end{equation}
where $M$ is the set of all finite arithmetic progressions in $\mathbb Z$.
\end{thm}
We now present a (noncommutative) version of this result on the group $\rSU2$.

\begin{thm}\label{Akylzhanov_1} If $2< p <\infty$ and $f\in L^p(\rSU2)$, then we have
\begin{equation}
\label{THM:necess_SU2}
	\sum_{l\in\frac12\mathbb{N}_0}(2l+1)^{p-2}\left(\sup_{\substack{k\in\frac12\mathbb{N}_0 \\ k\geq l}}\frac1{2k+1}\left|\Tr\widehat{f}(k)\right|\right)^p
\leq 
c\|f\|^p_{L^p(\rSU2)}
.
\end{equation}
\end{thm}
For completeness, we give a simple argument for Corollary \ref{COR:duality}.

\begin{proof}[Proof of Corollary \ref{COR:duality}]
The application of the duality of $L^p$ spaces yields
$$
\|f\|_{L^p(\rSU2)}
=
\sup_{\substack{g\in L^{p'}\\\|g\|_{L^{p^{\prime}}=1}}}
\left|
\int\limits_{\rSU2}f(x)\overline{g(x)}\,dx
\right|.
$$
Using Plancherel's identity \eqref{EQ:plancherel}, we get
$$
\int\limits_{\rSU2}f(x)\overline{g(x)}\,dx
=
\sum_{l\in\frac12\mathbb{N}_0}(2l+1)\Tr \widehat{f}(l)\widehat{g}(l)^*.
$$
It is easy to see that
\begin{eqnarray*}
(2l+1)=(2l+1)^{\frac52-\frac{4}{p}+\frac52-\frac{4}{p'}},\\
\left|\Tr \widehat{f}(l)\widehat{g}(l)^*\right|
\leq
\|\widehat{f}(l)\|_{\HS}
\|\widehat{g}(l)\|_{\HS}.
\end{eqnarray*}
Using these inequalities, applying H\"older inequality, 
for any $g\in L^{p'}$ with $\|g\|_{L^{p^{\prime}}}=1$, we have
\begin{multline*}
\left|
\sum_{l\in\frac12\mathbb{N}_0}(2l+1)\Tr {\widehat{f}(l)\widehat{g}(l)^*}
\right|
\leq
\sum_{l\in\frac12\mathbb{N}_0}(2l+1)^{\frac52-\frac{4}{p}}\|\widehat{f}(l)\|_{\HS}
(2l+1)^{\frac52-\frac{4}{p'}}\|\widehat{g}(l)\|_{\HS} \\
\leq
	\left(\sum_{l\in\frac12\mathbb{N}_0}(2l+1)^{\frac52 p-4}\|\widehat{f}(l)\|^p_{\HS}\right)^{\frac1p}
	\left(\sum_{l\in\frac12\mathbb{N}_0}(2l+1)^{\frac52 p'-4}\|\widehat{g}(l)\|^{p'}_{\HS}\right)^{\frac1{p'}}
\\
\leq
\left(\sum_{l\in\frac12\mathbb{N}_0}(2l+1)^{\frac52 p-4}\|\widehat{f}(l)\|^p_{\HS}\right)^{\frac1p} \|g\|_{L^{p^{\prime}}},
\end{multline*}
where we used Theorem \ref{Akylzhanov_2} in the last line.
Thus, we have just proved that
\begin{multline*}
\left|
\int\limits_{\rSU2}
f(x)\overline{g(x)}\,dx
\right|
\leq
\left|
\sum_{l\in\frac12\mathbb{N}_0}
(2l+1)\Tr \widehat{f}(l)\widehat{g}(l)^*
\right|
\\\leq
\left(\sum_{l\in\frac12\mathbb{N}_0}
(2l+1)^{\frac52 p-4}
\|\widehat{f}(l)\|^p_{\HS}
\right)^{\frac1p} 
\|g\|_{L^{p^{\prime}}}.
\end{multline*}
Taking supremum over all $g\in L^{p'}(\rSU2)$, we get \eqref{duality}.
This proves Corollary \ref{COR:duality}.
\end{proof}

\section{Lower bounds for Fourier multipliers on \texorpdfstring{$\rSU2$}{SU(2)}}
\label{main_results}

%

Let $A\colon C^{\infty}(\rSU2)\to \mathcal{D}'(\rSU2)$ be a continuous linear operator. 
Here we are concerned with left-invariant operators which means that 
$A\circ \tau_g=\tau_g\circ A$ for the left-translation 
$\tau_g f(x)=f(g^{-1} x)$. Using the Schwartz kernel theorem and the 
Fourier inversion formula one can prove that the left-invariant continuous operator 
$A$ can be written as a Fourier multiplier, namely, as
$$
 \widehat{Af}(l)=\sigma_{A}(l) \widehat{f}(l),
$$
for the symbol $\sigma_A(l)\in {\mathbb C}^{(2l+1)\times (2l+1)}$.
It follows from the Fourier inversion formula that we can write this also as
\begin{equation}\label{EQ:Ainv}
	A f(u)=\sum_{l\in\frac12\mathbb{N}_0}(2l+1)\Tr t^l(u)\sigma_A(l)\widehat{f}(l),
\end{equation}
where the symbol $\sigma_A(l)$ is given by 
$$
	\sigma_A(l)=t^l(e)^*At^l(e)=At^{l}(e),
$$
where $e$ is an identity matrix in $\rSU2$, and $(At^l)_{mk}=A(t^{l}_{mk})$ is defined
component-wise, for $-l\leq m,n\leq l$.
We refer to operators in these equivalent forms as (noncommutative) Fourier multipliers. 
The class of these operators on $\rSU2$ and their $L^{p}$-boundedness was
investigated in \cite{Coifman-Weiss:SU2-Argentina-1970,coifman+weiss_lnm},
and on general compact Lie groups in \cite{RuWi2013}. 
In particular, these authors proved H\"ormander--Mikhlin type multiplier theorems
in those settings, giving sufficient condition for the $L^p$-boundedness in terms of symbols. 
These conditions guarantee that the operator is of weak (1,1)-type which,
combined with a simple $L^{2}$-boundedness statement, implies the boundedness
on $L^{p}$ for all $1<p<\infty$. 

For a general (non-invariant) operator $A$, its matrix symbol $\sigma_A(u,l)$
will also depend on $u$. Such quantization \eqref{EQ:Ainv} has been consistently developed in
\cite{RT} and \cite{Ruzhansky+Turunen-IMRN}. We note that 
the $L^p$-boundedness results in \cite{RuWi2013} also cover such non-invariant
operators.

For a noncommutative Fourier multiplier $A$ we will write  
$A\in\mathnormal{M}^{q}_p(\rSU2)$ if $A$ extends to a bounded operator from $L^p(\rSU2)$ to 
$L^q(\rSU2)$. We introduce a norm $\|\cdot\|$ on $\mathnormal{M}^{q}_p(\rSU2)$ by setting
$$
	\|A\|_{\mathnormal{M}^{q}_p}:=\|A\|_{L^p\to L^q}.
$$
Thus, we are concerned with the question of what assumptions on the symbol $\sigma_{A}$
guarantee that $A\in\mathnormal{M}^{q}_{p}$. The sufficient conditions on 
$\sigma_A$ for $A\in\mathnormal{M}^p_p $ were investigated in \cite{RuWi2013}.
The aim of this section is to give a necessary condition on $\sigma_{A}$ for 
$A \in \mathnormal{M}^{q}_{p}$, for $1<p\leq 2 \leq q <\infty$. 

Suppose that $1<p\leq 2 \leq q<\infty$ and that 
$A\colon L^p(\rSU2)\to L^q(\rSU2)$ is a Fourier multiplier.    
The Plancherel identity \eqref{EQ:plancherel} implies that the operator $A$ is bounded from
$L^2(\rSU2)$ to $L^2(\rSU2)$ if and only if 
$\sup_{l}\|\sigma_A(l)\|_{op}<\infty$. 
Different other function spaces on the unitary dual have been discussed in \cite{RT}. 
Following Stein, we search for more subtle conditions on the symbols of 
noncommutative Fourier multipliers ensuring  their $L^{p}-L^{q}$ boundedness, and
we now prove a lower estimate which depends explicitly on $p$ and $q$. 

\begin{thm} \label{THM:lower-bound}
Let $1<p\leq 2 \leq q <\infty$ and let $A$ be a left-invariant operator on
$SU(2)$ such that $A\in \mathnormal{M}^q_p(\rSU2)$.  Then we have
\begin{eqnarray}
\label{SU_conds_1}
\sup_{l\in\frac12\mathbb{N}_0}
\frac{\min\limits_{n\in\{-l,\ldots,+l\}}|\sigma_A(l)_{nn}|}
{(2l+1)^{\frac1{p'}+\frac1{q}}}
\lesssim
\|A\|_{L^p(\rSU2)\to L^q(\rSU2)},
\\
\label{SU_conds_2}
\sup_{l\in\frac12\mathbb{N}_0}
\frac{|\Tr \sigma_A(l)|}
{(2l+1)^{1+\frac1{p'}+\frac1{q}}}
\lesssim
\|A\|_{L^p(\rSU2)\to L^q(\rSU2)}.
\end{eqnarray}
\end{thm}


One can see a similarity between \eqref{SU_conds_1}, \eqref{SU_conds_2} and 
\eqref{SU_conds-3} as 
\begin{equation}
\label{SU_conds-2}
\sup_{l\in\frac12\mathbb{N}_0}
\frac{1}{(2l+1)^{\frac1{p'}+\frac1{q}}}
\frac{1}{2l+1} 
{\left|\Tr \sigma_A(l)\right|}
\lesssim
\|A\|_{L^p(\rSU2)\to L^q(\rSU2)}.
\end{equation}
We also note that estimates \eqref{SU_conds_1} and \eqref{SU_conds_2} can not be
immediately compared because the value of the trace in \eqref{SU_conds_2}
depends on the signs of the diagonal entries of $\sigma_A(l)$.

\begin{proof}[Proof of Theorem \ref{THM:lower-bound}]
In \cite{Travaglini1980} it was proven that for any $l\in \frac12\mathbb{N}_0$ there exists a basis
for $t^l\in \SU2$ and a diagonal matrix coefficient $t^l_{nn}$ 
(i.e. for some $n$, $-l\leq n\leq l$),
such that
\begin{equation}
\label{repr_elem_estim}
	\|t^l_{nn}\|_{L^p(\rSU2)}\cong \frac1{(2l+1)^{\frac1p}}.
\end{equation}
Now, we use this result to establish a lower bound for the norm of $A\in M^q_p(\rSU2)$.
Let us fix an arbitrary $l_{0}\in\frac12\NN_{0}$ and the corresponding diagonal element $t^{l_0}_{nn}$. 
We consider $f_{l_0}(g)$ such that its matrix-valued Fourier coefficient
\begin{equation}
\label{choice_of_f}
	\widehat{f_{l_0}}(l)
=\diag(0,\ldots, 1,0,\ldots)\delta^{l}_{l_0}
\end{equation}
has only one non-zero diagonal coefficient $1$ at the $n^{th}$ diagonal entry.
Then by the Fourier inversion formula we get
$ f_{l_0}(g)=(2l_0+1) t^{l_0}_{nn}(g).$
 By definition, we get
\begin{multline*}
\|A\|_{L^p\to L^q} = \sup_{f\neq 0}
\frac{\|\sum_{l\in\frac12\mathbb{N}_0}(2l+1)
\Tr t^l(u)\sigma_A(l)\widehat{f}(l) \|_{L^q(\rSU2)}}
{\|f\|_{L^p(\rSU2)}}\\
\geq
\frac{\|\sum_{l\in\frac12\mathbb{N}_0}(2l+1)
\Tr t^l(u) \sigma_A(l)\widehat{f_{l_0}}(l)\|_{L^q(\rSU2)}}
{\|f_{l_0}\|_{L^p(\rSU2)}}.
\end{multline*}
Recalling \eqref{choice_of_f}, we get
$$
\|A\|_{L^p\to L^q}
\gtrsim
\frac{\|(2l_0+1)\Tr{ t^{l_0}(g) \sigma_A(l_0)\widehat{f_{l_0}}(l)}\|_{L^q(\rSU2)}}{\|f_{l_0}\|_{L^p(\rSU2)}}.
$$
Setting $h(g):=(2l_0+1)\Tr{ t^{l_0}(g) \sigma_A(l_0)\widehat{f_{l_0}}(l_0)}$, we have 
$\widehat{h}(l)=0$ for $l\not=l_{0}$, and 
$\widehat{h}(l_{0})=\sigma_A(l_0)\widehat{f_{l_0}}(l_0)$.
Consequently, we get 
$$
\sup_{\substack{k\in\frac12\mathbb{N}_0\\k\geq l}}\frac1{2k+1}\left|\Tr \widehat{h}(k)\right|
=
\begin{cases}
0,& l > l_0,\\
\frac1{2l_0+1}\left|\sigma_A(l_0)_{nn}\right|,& 1\leq l\leq l_0.
\end{cases}
$$
Using this, Theorem \ref{Akylzhanov_1} and \eqref{repr_elem_estim}, we have
\[
\|A\|_{L^p\to L^q}
\gtrsim
\frac{\left(\displaystyle\sum^{l_0}_{l=1}(2l+1)^{q-2}\left(\frac1{2l_0+1}\left|\sigma_A(l_0)_{nn}\right|\right)^q\right)^{\frac1q}}
{(2l_0+1)^{1-\frac1p}},
\]
where $l_0$ is an arbitrary fixed half-integer. Direct calculation now shows that
\begin{align*}
	\frac{\left(\displaystyle\sum^{l_0}_{l=1}(2l+1)^{q-2}\left(\frac1{2l_0+1}\left|\sigma_A(l_0)_{nn}\right|\right)^q\right)^{\frac1q}}
{(2l_0+1)^{1-\frac1p}}
=
\frac1{2l_0+1}\left|\sigma_A(l_0)_{nn}\right|
\frac{\left(\displaystyle\sum^{l_0}_{l=1}(2l+1)^{q-2}\right)^{\frac1q}}
{(2l_0+1)^{1-\frac1p}}
\\
=
\frac1{2l_0+1}\left|\sigma_A(l_0)_{nn}\right|
\frac{(2l_0+1)^{1-\frac1q}}
{(2l_0+1)^{1-\frac1p}}
\cong
\frac{\left|\sigma_A(l_0)_{nn}\right|}
{(2l_0+1)^{\frac1{p'}+\frac1{q}}}.
\end{align*}
Taking infimum over all  $n\in\{-l_0,-l_0+1,\ldots,l_0-1,l_0\}$ and then supremum over all half-integers, we have
$$
\|A\|_{L^p\to L^q}
\gtrsim
\sup_{l\in\frac12\mathbb{N}_0}
\frac{\min\limits_{n\in\{-l,\ldots,+l\}}|\sigma_A(l)_{nn}|}
{(2l+1)^{\frac1{p'}+\frac1{q}}}.
$$
This proves estimate \eqref{SU_conds_1}.
Now, we will prove estimate \eqref{SU_conds_2}.
Let us fix some $l_{0}\in\frac12\NN_{0}$ and consider now $f_{l_0}(u):=(2l_0+1) \chi_{l_0}(u)$,
where $\chi_{l_0}(u)=\Tr t^{l_{0}}(u)$ is the character of the representation $t^{l_{0}}$.
Then, in particular, we have 
\begin{equation}\label{EQ:char1}
\widehat{f_{l_0}}(l)=
\begin{cases}
I_{2l+1}, & l=l_0,\\
0,& l \neq l_0,\\
\end{cases}
\end{equation}
 where $I_{2l+1}\in {\mathbb C}^{(2l+1)\times (2l+1)}$ is the identity matrix. 
 Using the Weyl character formula, we can write
 $$
\chi_{l_0}(u) = \sum^{l_0}_{k=-l_0}e^{ikt},
$$
where
$
u=v^{-1}\begin{pmatrix}e^{it} & 0 \\ 0 & e^{-it}\end{pmatrix}v.
$
The value of $\chi_{l_0}(u)$ does not depend on $v$ since characters are central.
Further,
the application of the Weyl integral formula yields
$$
\|f_{l_0}\|_{L^p(SU(2))}= (2l_0+1)\|\chi_{l_0}\|_{L^p(SU(2))}=
(2l_0+1)
\left(\int\limits^{2\pi}_0 \left|\sum^{l_0}_{k=-l_0}e^{ikt}\right|^p2\sin^2 t\,
\frac{dt}{2\pi}\right)^{\frac1p}.
$$ 	
It is clear that
$\left|e^{i(-l_0-1)t}\sum^{l_0}_{k=-l_0}e^{i(k+l_0+1)t} \right|=\left|\sum^{2l_0+1}_{k=1}e^{ikt}\right|$.
We call $D_{2l_0+1}(t):=\sum^{2l_0+1}_{k=1}e^{ikt}$ the Dirichlet kernel. 
Then, we apply 
\cite[Corollary 4]{NED}
to the Dirichlet kernel $D_{2l_0+1}(t)$, to get
\begin{equation}
\|\chi_{l_0}\|_{L^p(\rSU2)}
\lesssim
\|D_{2l_0+1}\|_{L^p(0,2\pi)}
\cong
(2l_0+1)^{1-\frac1{p}}.
\end{equation}
By definition, we get
\begin{multline*}
\|A\|_{L^p\to L^q} = \sup_{f\neq 0}
\frac{\|\sum_{l\in\frac12\mathbb{N}_0}(2l+1)
\Tr{ t^l(u)\sigma_A(l)\widehat{f}(l)} \|_{L^q(\rSU2)}}{\|f\|_{L^p(\rSU2)}}\\
\geq
\frac{\|\sum_{l\in\frac12\mathbb{N}_0}(2l+1)
\Tr{ t^l(u) \sigma_A(l)\widehat{f_{l_0}}(l) }\|_{L^q(\rSU2)}}{\|f_{l_0}\|_{L^p(\rSU2)}}.
\end{multline*}
Recalling \eqref{EQ:char1}, we obtain
$$
\|A\|_{L^p\to L^q}
\gtrsim
\frac{\|(2l_0+1)\Tr{ t^{l_0}(g) \sigma_A(l_0)}\|_{L^q(\rSU2)}}{\|f_{l_0}\|_{L^p(\rSU2)}}.
$$
Setting $h(g):=(2l_0+1)\Tr{ t^{l_0}(g) \sigma_A(l_0)}$, we have 
$\widehat{h}(l)=0$ for $l\not=l_{0}$, and 
$\widehat{h}(l_{0})=\sigma_A(l_0)$.
Consequently, we get 
$$
\sup_{\substack{k\in\frac12\mathbb{N}_0\\k\geq l}}\frac1{2k+1}\left|\Tr \widehat{h}(k)\right|
=
\begin{cases}
0,& l > l_0,\\
\frac1{2l_0+1}\left|\Tr \sigma_A(l_0)\right|,& 1\leq l\leq l_0.
\end{cases}
$$
Using this and Theorem \ref{Akylzhanov_1}, we have
\[
\|A\|_{L^p\to L^q}
\gtrsim
\frac{\left(\displaystyle\sum^{l_0}_{l=1}(2l+1)^{q-2}\left(\frac1{2l_0+1}\left|\Tr \sigma_A(l_0)\right|\right)^q\right)^{\frac1q}}
{(2l_0+1)(2l_0+1)^{1-\frac1p}},
\]
where $l_0$ is an arbitrary fixed half-integer. Direct calculation shows that
\begin{align*}
	\frac{\left(\displaystyle\sum^{l_0}_{l=1}(2l+1)^{q-2}\left(\frac1{2l_0+1}\left|\Tr \sigma_A(l_0)\right|\right)^q\right)^{\frac1q}}
{(2l_0+1)(2l_0+1)^{1-\frac1p}}
=
\frac1{2l_0+1}\left|\Tr \sigma_A(l_0)\right| 
\frac{\left(\displaystyle\sum^{l_0}_{l=1}(2l+1)^{q-2}\right)^{\frac1q}}
{(2l_0+1)(2l_0+1)^{1-\frac1p}}
\\
=
\frac1{2l_0+1}\left|\Tr \sigma_A(l_0)\right| 
\frac{(2l_0+1)^{1-\frac1q}}
{(2l_0+1)(2l_0+1)^{1-\frac1p}}
\cong
\frac{\left|\Tr \sigma_A(l_0)\right|}
{(2l_0+1)^{1+\frac1{p'}+\frac1{q}}}.
\end{align*}
Taking supremum over all half-integers, we have
$$
\|A\|_{L^p\to L^q}
\gtrsim
\sup_{l\in\frac12\mathbb{N}_0}
\frac{\left|\Tr \sigma_A(l)\right|}
{(2l+1)^{1+\frac1{p'}+\frac1{q}}}.
$$
This proves the estimate \eqref{SU_conds_2}
%
\end{proof}

\section{Upper bounds for Fourier multipliers on \texorpdfstring{$\rSU2$}{SU(2)}}
\label{SEC:upper-bounds}

In this section we give a noncommutative $\rSU2$ analogue of
the upper bound for Fourier multipliers, analogous to the one on the circle
$\TT$ in Theorem \ref{THM:NT}
(see also \cite{Nursultanov:MZ-1998,Nursultanov-Tikhonov:JGA-2011} for the circle case).

\begin{thm}\label{THM:upper}
If $1<p\leq 2\leq q<\infty$ and $A$ is a left-invariant operator on $\rSU2$,
then we have
\begin{equation}
\label{EQ:upper}
	\|A\|_{L^p(\rSU2)\to L^q(\rSU2)}
\lesssim
\sup_{s>0}s\left( \sum_{\substack{l\in\frac12\mathbb{N}_0 \\ \|\sigma_A(l)\|_{\op}>s}} (2l+1)^2\right)^{\frac1p-\frac1{q}}.	
\end{equation}
\end{thm}

\begin{proof}
Since $A$ is a left-invariant operator, it acts on $f$ via the multipication of $\widehat{f}$ by the symbol $\sigma_A$
\begin{equation}
\widehat{Af}(\pi)=\sigma_A(\pi)\widehat{f}(\pi),
\end{equation}
where 
$$
\sigma_A(\pi)=\pi(x)^*A\pi(x)\big|_{x=e}.
$$

Let us first assume that $p\leq q'$.
Since $q'\leq 2$, for $f\in C^{\infty}(\rSU2)$ the Hausdorff-Young inequality gives 
\begin{align}
\label{EQ:Af-norm}
\begin{split}
\|Af\|_{L^q(\rSU2)}
&
\leq
\|\widehat{Af}\|_{\ell^{q'}(\SU2)}
=
\|\sigma_A\widehat{f}\|_{\ell^{q'}(\SU2)}
\\
&=
\left(
\sum\limits_{l\in\SU2}(2l+1)^{2-\frac{q'}2}
\|\sigma_A(l)\widehat{f}(l)\|^{q'}_{\HS}
\right)^{\frac1{q'}}
\\&\leq
\left(
\sum\limits_{l\in\SU2}(2l+1)^{2-\frac{q'}2}
\|\sigma_A(l)\|^{q'}_{\op}\|\widehat{f}(l)\|^{q'}_{\HS}
\right)^{\frac1{q'}}
.
\end{split}
\end{align}
The case $q'\leq (p')'$ can be reduced to the case $p\leq q'$ as follows.
The application of Theorem \ref{THM:adjoint-operator} with $G=\rSU2$ and $\mu=\{\text{Haar measure on $\rSU2$}\}$ yields
\begin{equation}
\label{EQ:A-A-star-norm}
\|A\|_{L^p(\rSU2)\to L^q(\rSU2)}
=
\|A^*\|_{L^{q'}(\rSU2)\to L^{p'}(\rSU2)}.
\end{equation}
The symbol $\sigma_{A^*}(l)$ of the adjoint operator $A^*$ equals to $\sigma_{A}^*(l)$
\begin{equation}
\label{EQ:symbol-conjugate}
\sigma_{A^*}(l)=\sigma_A^*(l),\quad l\in\frac12\NN_0,
\end{equation}
and its operator norm $\|\sigma_{A^*}(l)\|_{\op}$ equals to $\|\sigma_{A}(l)\|_{\op}$.
Now, we are in a position to apply Corollary \ref{Cor:general_Paley_inequality}. 
Set $\frac1r=\frac1p-\frac1q$. We observe that
with $\sigma(t^l):=\|\sigma_A(t^l)\|^r_{op} I_{2l+1},l\in\frac12\NN_0$ and $b=q'$, the assumptions of Corollary \ref{Cor:general_Paley_inequality} are satisfied and  we obtain
\begin{align}
\begin{split}
\left(
\sum\limits_{l\in\SU2}(2l+1)^{2-\frac{q'}2}
\|\sigma_A(l)\|^{q'}_{\op}\|\widehat{f}(l)\|^{q'}_{\HS}
\right)^{\frac1{q'}}
\\\lesssim
\left(
\sup_{s>0}s\sum\limits_{\substack{t^l\in\SU2\\\|\sigma(t^l)\|^r_{\op}>s}}(2l+1)^2
\right)^{\frac1r}
\|f\|_{L^p(\rSU2)},\quad f\in L^p(\rSU2),
\end{split}
\end{align}
in view of $\frac1{q'}-\frac1{p'}=\frac1p-\frac1q=\frac1r$.
Thus, for $1<p\leq 2 \leq q <\infty$, we obtain
\begin{equation}
\|Af\|_{L^q(\rSU2)}
\lesssim
\left(
\sup_{s>0}s\sum\limits_{\substack{t^l\in\SU2\\\|\sigma(t^l)\|^r_{\op}>s}}(2l+1)^2
\right)^{\frac1r}
\|f\|_{L^p(\rSU2)}.
\end{equation}
Further, it can be easily checked that
\begin{multline*}
\left(
\sup_{s>0}s\sum\limits_{\substack{t^l\in\SU2\\\|\sigma(t^l)\|_{op}>s}}(2l+1)^2
\right)^{\frac1r}
=
\left(
\sup_{s>0}s\sum\limits_{\substack{t^l\in\SU2\\\|\sigma_A(t^l)\|_{\op}>s^{\frac1r}}}(2l+1)^2
\right)^{\frac1r}
\\=
\left(
\sup_{s>0}s^{r}\sum\limits_{\substack{t^l\in\SU2\\\|\sigma_A(t^l)\|_{\op}>s}}(2l+1)^2
\right)^{\frac1r}
=
\sup_{s>0}s \left(\sum\limits_{\substack{t^l\in\SU2\\\|\sigma_A(t^l)\|_{\op}>s}}(2l+1)^2\right)^{\frac1r}.
\end{multline*}
This completes the proof.
\end{proof}
For the completness, we give a short proof of Theorem \ref{THM:adjoint-operator} used in the proof.
\begin{thm} 
\label{THM:adjoint-operator}
Let $(X,\mu)$ be a measure space and $1<p,q<\infty$. Then we have 
\begin{equation}
\|A\|_{L^p(X,\mu)\to L^q(X,\mu)}
=
\|A^*\|_{L^{q'}(X,\mu)\to L^{p'}(X,\mu)},
\end{equation}
where $A^*\colon L^{q'}(X,\mu)\to L^{p'}(X,\mu)$ is the adjoint of $A$.
\end{thm}
\begin{proof}[Proof of Theorem \ref{THM:adjoint-operator}]
Let $f\in L^p\cap L^2$ and $g\in L^{q'}\cap L^2$. By H\"older inequality, we have
\begin{equation}
\left|(Af,g)_{L^2}\right|
=
\left|(A^*g,f)_{L^2}\right|
\leq
\|A^*g\|_{L^{p'}}\|f\|_{L^{p}}
\leq
\|A^*\|_{L^{q'}\to L^{p'}}
\|g\|_{L^{q'}}
\|f\|_{L^{p}}.
\end{equation}
Thus, we  get
\begin{equation}
\label{EQ:A-A-star}
\|A\|_{L^p\to L^q}
\leq
\|A^*\|_{L^{q'}\to L^{p'}}.
\end{equation}
Analogously, we show that
\begin{equation}
\label{EQ:A-star-A}
\|A^*\|_{L^{q'}\to L^{p'}}
\leq
\|A\|_{L^p\to L^q}.
\end{equation}
The combination of \eqref{EQ:A-A-star} and \eqref{EQ:A-star-A} yields
$$
\|A\|_{L^p\to L^q}
=
\|A^*\|_{L^{q'}\to L^{p'}}.
$$
This completes the proof.
\end{proof}

\section{Proofs of Theorems from Section \ref{SEC:Hardy_Littlewood_Paley_inequalities}}
\label{SEC:Hardy_Littlewood_Paley_inequalities_Proofs}

\begin{proof}[Proof of Theorem \ref{THM:Paley_inequality}]

Let $\mu$ give measure $\|\sigma(t^l)\|^2_{op}(2l+1)^2,l\in\frac12\NN_0$ to the set consisting of the single point $\{t^l\}, t^l\in\SU2$, 
and measure zero to a set which does not contain any of these points, i.e.
$$
	\mu\{t^l\}:=\|\sigma(t^l)\|^2_{op}(2l+1)^2.
$$
We define the space $L^p(\SU2,\mu)$, $1\leq p<\infty$, 
as the space of complex (or real) sequences
$a=\{a_{l}\}_{l\in\frac12\NN_0}$ such that
\begin{equation}\label{EQ:Lpmu}
\|a\|_{L^p(\SU2,\mu)}
:=
\left(
\sum\limits_{\substack{l\in\frac12\NN_0}}
|a_{l}|^p
\|\sigma(t^l)\|^2_{op}(2l+1)^2
\right)^{\frac1p}
<\infty.
\end{equation}
We will show that the sub-linear operator
$$
	A\colon L^p(\rSU2)\ni f \mapsto Af=
	\left\{\frac{\|\widehat{f}(t^l)\|_{\HS}}{\sqrt{2l+1}\|\sigma(t^l)\|_{op}}\right\}_{t^l\in\SU2}\in L^p(\SU2,\mu)
$$
is well-defined and bounded from $L^p(\rSU2)$ to $L^p(\SU2,\mu)$ for $1<p\leq 2$. 
In other words, we claim that we have the estimate
\begin{multline}
\label{Paley_inequality_alt}
	\|Af\|_{L^p(\SU2,\mu)}
=
\left(
\sum\limits_{\substack{t^l\in\SU2}}
\left(
\frac{\|\widehat{f}(t^l)\|_{\HS}}{\sqrt{2l+1}\|\sigma(t^l)\|_{op}}
\right)^p
\|\sigma(t^l)\|^2_{op}(2l+1)^2
\right)^{\frac1p}
\\\lesssim
K^{\frac{2-p}{p}}_{\sigma}
\|f\|_{L^p(\rSU2)},
\end{multline}
which would give \eqref{EQ:Paley_inequality}
and where we set $K_{\sigma}:=\sup_{s>0}s\sum\limits_{\substack{t^l\in\SU2\\ \|\sigma(t^l)\|_{op}\geq s}}(2l+1)^2$. 
We will show that $A$ is of weak type (2,2) and of weak-type (1,1). For definition and discussions we refer to Section \ref{SEC:Marc_Interpol_Theorem} where we give definitions of weak-type, formulate and prove Marcinkiewicz interpolation Theorem \ref{THM:Marc-Interpol-Un-Dual}.
More precisely, with the distribution function $\nu$ as in Theorem \ref{THM:Marc-Interpol-Un-Dual},
we show that
\begin{eqnarray}
\label{EQ:THM:Paley_inequality_weak_1}
\nu_{\SU2}(y;Af)
&\leq &
\left(\frac{M_2\|f\|_{L^2(\rSU2)}}{y}\right)^2  \quad\text{with norm } M_2 = 1,\\
\label{EQ:THM:Paley_inequality_weak_2}
\nu_{\SU2}(y;Af)
&\leq &
\frac{M_1\|f\|_{L^1(\rSU2)}}{y}  \qquad\text{with norm } M_1 = K_{\sigma}.
\end{eqnarray}
Then \eqref{Paley_inequality_alt} would follow by Marcinkiewicz interpolation 
Theorem \ref{THM:Marc-Interpol-Un-Dual}. 
Now, to show \eqref{EQ:THM:Paley_inequality_weak_1}, using Plancherel's identity \eqref{EQ:plancherel},
 we get
\begin{multline*}
		y^2
		\nu_{\SU2}(y;Af)
		\leq
	\|Af\|^2_{L^p(\SU2,\mu)}
:=
\sum\limits_{\substack{t^l\in\SU2}}
\left(
\frac{\|\widehat{f}(t^l)\|_{\HS}}{\sqrt{2l+1}\|\sigma(t^l)\|_{op}}
\right)^2
\|\sigma(t^l)\|^2_{op}(2l+1)^2
\\=
\sum\limits_{\substack{t^l\in\SU2}}
(2l+1)
\|\widehat{f}(t^l)\|^2_{\HS}
=
\|\widehat{f}\|^2_{\ell^2(\SU2)}
=
\|f\|^2_{L^2(\rSU2)}.
\end{multline*}
Thus, $A$ is of type (2,2) with norm $M_2\leq1$.
Further, we show that $A$ is of weak-type (1,1) with norm $M_1=C$; more precisely, we show that
\begin{equation}
\label{weak_type}
\nu_{\SU2}\{t^l\in\SU2 \colon \frac{\|\widehat{f}(t^l)\|_{\HS}}{\sqrt{2l+1}\|\sigma(t^l)\|_{op}} > y\}
\lesssim
K_{\sigma}\,
\dfrac{\|f\|_{L^1(\rSU2)}}{y}.
\end{equation}
The left-hand side here is the weighted sum $\sum \|\sigma(t^l)\|^2_{op}(2l+1)^2$ taken over those $t^l\in\SU2$ for which $\dfrac{\|\widehat{f}(t^l)\|_{\HS}}{\sqrt{2l+1}\|\sigma(t^l)\|_{op}}>y$. From the definition of the Fourier transform  it follows that
$$
	\|\widehat{f}(t^l)\|_{\HS}\leq \sqrt{2l+1}\|f\|_{L^1(\rSU2)}.
$$
Therefore, we have
$$
y<\frac{\|\widehat{f}(t^l)\|_{\HS}}{\sqrt{2l+1}\|\sigma(t^l)\|_{op}}
\leq
\frac{\|f\|_{L^1(\rSU2)}}{\|\sigma(t^l)\|_{op}}.
$$
Using this, we get
$$
\left\{
t^l\in\SU2\colon 
\frac{\|\widehat{f}(t^l)\|_{\HS}}{\sqrt{2l+1}\|\sigma(t^l)\|_{op}}>y
\right\}
\subset
\left\{
t^l\in\SU2\colon 
\frac{\|f\|_{L^1(\rSU2)}}{\|\sigma(t^l)\|_{op}}>y
\right\}
$$
for any $y>0$. Consequently,
$$
\mu\left\{
t^l\in\SU2\colon 
\frac{\|\widehat{f}(t^l)\|_{\HS}}{\sqrt{2l+1}\|\sigma(t^l)\|_{op}}>y
\right\}
\leq
\mu\left\{
t^l\in\SU2\colon 
\frac{\|f\|_{L^1(\rSU2)}}{\|\sigma(t^l)\|_{op}}>y
\right\}.
$$
Setting $v:=\frac{\|f\|_{L^1(\rSU2)}}{y}$, we get

\begin{equation}
\label{PI_intermed_est_1}
\mu\left\{
t^l\in\SU2\colon 
\frac{\|\widehat{f}(t^l)\|_{\HS}}{\sqrt{2l+1}\|\sigma(t^l)\|_{op}}>y
\right\}
\leq
\sum\limits_{\substack{t^l\in\SU2 \\ \|\sigma(t^l)\|_{op}\leq v}}
\|\sigma(t^l)\|^2_{op}(2l+1)^2.
\end{equation}
We claim that
\begin{equation}\label{EQ:aux1}
\sum\limits_{\substack{t^l\in\SU2 \\ \|\sigma(t^l)\|_{op}\leq v}}
\|\sigma(t^l)\|^2_{op}(2l+1)^2
\lesssim 
K_{\sigma}
v.
\end{equation}
In fact, we have
$$
	\sum\limits_{\substack{t^l\in\SU2 \\ \|\sigma(t^l)\|_{op}\leq v}}
\|\sigma(t^l)\|^2_{op}(2l+1)^2
=
	\sum\limits_{\substack{t^l\in\SU2 \\ \|\sigma(t^l)\|_{op}\leq v}}
(2l+1)^2\int\limits^{\|\sigma(t^l)\|^2_{op}}_0 d\tau.
$$
We can interchange sum and integration to get
$$
	\sum\limits_{\substack{t^l\in\SU2 \\ \|\sigma(t^l)\|_{op}\leq v}}
(2l+1)\int\limits^{\|\sigma(t^l)\|^2_{op}}_0 d\tau
=
\int\limits^{v^2}_0 d\tau \sum
\limits_{\substack{t^l\in\SU2 \\ \tau^{\frac12}\leq \|\sigma(t^l)\|_{op}\leq v}}
(2l+1)^2.
$$
Further, we make a substitution $\tau=s^2$, yielding
\begin{multline*}
\int\limits^{v^2}_0 d\tau \sum\limits_{\substack{t^l\in\SU2 \\ 
\tau^{\frac12}\leq \|\sigma(t^l)\|_{op}\leq v}}(2l+1)^2
=
2\int\limits^{v}_0 s\,ds 
\sum\limits_{\substack{t^l\in\SU2 \\ s \leq \|\sigma(t^{l})\|_{op}\leq v}}(2l+1)^2
\\
\leq
2\int\limits^{v}_0 s\,ds 
\sum\limits_{\substack{t^l\in\SU2 \\ s \leq \|\sigma(t^l)\|_{op}}}(2l+1)^2.
\end{multline*}
Since 
$$
	s\sum\limits_{\substack{t^l\in\SU2 \\ s \leq \|\sigma(t^l)\|_{op} } } (2l+1)^2
	\leq 
\sup_{s>0}	s\sum\limits_{\substack{t^l\in\SU2 \\ s \leq \|\sigma(t^l)\|_{op} } } (2l+1)^2
=:K_{\sigma}
$$
is finite by the definition of $K_{\sigma}$, we have
$$
	2\int\limits^{v}_0 s\,ds 
\sum\limits_{\substack{t^l\in\SU2 \\ s \leq \|\sigma(t^l)\|_{op}}}(2l+1)^2
\lesssim K_{\sigma} v.
$$
This proves \eqref{EQ:aux1}.
We have just proved inequalities
\eqref{EQ:THM:Paley_inequality_weak_1},
\eqref{EQ:THM:Paley_inequality_weak_2}.
Then by using \\Marcinkiewicz' interpolation theorem (Theorem \ref{THM:Marc-Interpol-Un-Dual} from Section \ref{SEC:Marc_Interpol_Theorem})  with $p_1=1, p_2=2$ and 
$\frac1p=1-\theta+\frac{\theta}2$ we now obtain
\begin{multline*}
\left(
\sum\limits_{\substack{l\in\frac12\NN_0}}
\left(
\frac{\|\widehat{f}(\pi)\|_{\HS}}{\sqrt{2l+1}\|\sigma(\pi)\|_{op}}
\right)^p
\|\sigma(\pi)\|^2_{op}(2l+1)^2
\right)^{\frac1p}
\\ =
\|Af\|_{L^p(\SU2,\mu)}
\lesssim
K^{\frac{2-p}{p}}_{\sigma}
\|f\|_{L^p(\rSU2)}.
\end{multline*}
This completes the proof.

\end{proof}

Now we prove the Hardy--Littlewood type inequality given in
Theorem \ref{Akylzhanov_2}.
\begin{proof}[Proof of Theorem \ref{Akylzhanov_2}]
Let $\nu$ give measure $\dfrac1{(2l+1)^4}$ to the set consisting of the single point $l, l=0,\frac12,1,\frac32,2,\ldots$, and measure zero to a set which does not contain any of these points. We will show that the sub-linear operator
$$
Tf := \{(2l+1)^{\frac52}\|\widehat{f}(l)\|_{\HS}\}_{l\in\frac12\mathbb{N}_0} 
$$
is well-defined and bounded from $L^p(\rSU2)$ to $L^p(\frac12\mathbb{N}_0, \nu)$ for 
$1<p\leq 2$, with 
$$
	\|Tf\|_{L^p(\SU2,\nu)}=\left(\sum_{l\in\frac12\mathbb{N}_0}\left( (2l+1)^{\frac52}
	\|\widehat{f}(l)\|_{\HS}\right)^p\cdot(2l+1)^{-4}\right)^{\frac1p}.
$$
This will prove Theorem \ref{Akylzhanov_2}.
We first show that $T$ is of type $(2,2)$ and weak type $(1,1)$.
Using Plancherel's identity \eqref{EQ:plancherel}, we get
\begin{multline*}
	\|Tf\|^2_{L^p(\SU2,\nu)} =
\sum_{l\in \frac12\mathbb{N}_0}(2l+1)^{\frac{5p}2-4}\|\widehat{f}(l)\|^2_{\HS}
=
\sum_{l\in \frac12\mathbb{N}_0}(2l+1)\|\widehat{f}(l)\|^2_{\HS}
\\=
\|\widehat{f}\|^2_{\ell^2(\SU2)}
=\|f\|^2_{L^2(\rSU2)}.
\end{multline*}
Thus, $T$ is of type $(2,2)$.

Further, we show that $T$ is of {\it weak type} (1,1); more precisely we show that
\begin{equation}
\label{weak_estimate}
\nu\{l\in\frac12\mathbb{N}_0\colon (2l+1)^{\frac52}\|\widehat{f}(l)\|_{\HS}>y\}\leq \frac{4}{3}\frac{\|f\|_{L^1(\rSU2)}}{y}.
\end{equation}
The left-hand side here is the sum 
$\displaystyle\sum\frac1{(2l+1)^4}$ taken over those $l\in\frac12\mathbb{N}_0$ for which
$(2l+1)^{\frac52}\|\widehat{f}(l)\|_{\HS}>y$. 
From the definition of the Fourier transform it follows that
$$
	\|\widehat{f}(l)\|_{\HS}\leq \sqrt{2l+1}\|f\|_{L^1(\rSU2)}.
$$
Therefore, we have
$$
y<(2l+1)^{\frac52}\|\widehat{f}(l)\|_{\HS}\leq (2l+1)^{\frac52+\frac12}\|f\|_{L^1(\rSU2)}.
$$
Using this, we get
$$
	\left\{l\in\frac12\mathbb{N}_0 \colon (2l+1)^{\frac52}\|\widehat{f}(l)\|_{\HS}>y \right\}
\subset
\left\{l\in\frac12\mathbb{N}_0 \colon (2l+1)>\left(\frac{y}{\|f\|_{L^1}}\right)^{\frac1{3}}\right\}
$$
for any $y>0$.
Consequently,
$$
\nu\left\{l\in\frac12\mathbb{N}_0 \colon (2l+1)^{\frac52}\|\widehat{f}(l)\|_{\HS}>y \right\}
\leq
\nu\left\{l\in\frac12\mathbb{N}_0 \colon (2l+1)>\left(\frac{y}{\|f\|_{L^1}}\right)^{\frac1{3}}\right\}.
$$
We set $w:=
\left( \frac{y}{\|f\|_{L^1(\rSU2)}}\right)^{\frac1{3}}$. Now, we estimate $\nu\left\{l\in\frac12\mathbb{N}_0 \colon (2l+1)>w\right\}$. By definition, we have
$$
\nu\left\{
l\in\frac12\mathbb{N}_0 \colon (2l+1)>\left(\frac{y}{\|f\|_{L^1}}\right)^{\frac1{3}}
\right\}
=
	\sum^{\infty}_{n>w}\frac1{n^{4}}.
$$
In order to estimate this series, we introduce the following lemma.
\begin{lem} Suppose $\beta>1$ and $w>0$. Then we have
\begin{equation}
	\sum^{\infty}_{n>w}\frac1{n^{\beta}}
\leq
\begin{cases}
\frac{\beta}{\beta-1},& w\leq 1, \\
\frac1{\beta-1}\frac1{w^{\beta-1}},& w > 1.
\end{cases}
\end{equation}
\end{lem}
The proof is rather straightforward.
Now, suppose $w\leq 1$. Then applying this lemma with $\beta=4$, we have
$$
	\sum^{\infty}_{n>w}\frac1{n^{4}}
\leq
	\frac{4}{3}.
$$
Since 
$
	1 \leq \frac1{w^{3}},
$
we obtain
$$
	\sum^{\infty}_{n>w}\frac1{n^{4}}
\leq
	\frac{4}{3}
\leq
\frac{4}{3}\frac1{w^{3}}.
$$
Recalling that $w=\left(\frac{y}{\|f\|_{L^1(\rSU2)}}\right)^{\frac1{3}}$, we finally obtain
$$
\nu\left\{
l\in\frac12\mathbb{N}_0 \colon (2l+1)>\left(\frac{y}{\|f\|_{L^1}}\right)^{\frac1{3}}
\right\}
=
	\sum^{\infty}_{n>w}\frac1{n^{4}}
\leq
	\frac{4}{3} \frac{\|f\|_{L^1(\rSU2)}}{y}.
$$
Now, if $w>1$, then we have
$$
	\sum^{\infty}_{n>w}\frac1{n^{4}} 
\leq
\frac1{3}\frac1{w^{3}}=\frac{4}{3}\frac{\|f\|_{L^1}}{y}.
$$
Finally, we get
$$
\nu\left\{
l\in\frac12\mathbb{N}_0 \colon (2l+1)>\left(\frac{y}{\|f\|_{L^1}}\right)^{\frac1{3}}
\right\}
\leq
\frac{4}{3}\frac{\|f\|_{L^1(\rSU2)}}{y}.
$$
This proves \eqref{weak_estimate}.

By Marcinkiewicz interpolation Theorem \ref{THM:Marc-Interpol-Un-Dual} with $p_1=1, 
p_2=2$, we obtain
$$
	\left(
\sum_{l\in\frac12\mathbb{N}_0}(2l+1)^{\frac52p-4}\|\widehat{f}(l)\|^p_{\HS}
\right)^{\frac1p}
=
\|Tf\|_{L^p(\SU2,\nu)}
\leq
c_p \|f\|_{L^p(\rSU2)}.
$$
This completes the proof of Theorem \ref{Akylzhanov_2}.
.
\end{proof}

Now we prove Theorem \ref{Akylzhanov_1}.

\begin{proof}[Proof of Theorem \ref{Akylzhanov_1}]
We first simplify the expression for $\Tr\widehat{f}(k)$.
By definition, we have
$$
	\widehat{f}(k) = \int\limits_{\rSU2}f(u) T^k(u)^*\,du,\quad k\in\frac12\mathbb{N}_0,
$$
where $T^k$ is a finite-dimensional representation of $\SU2$ as in 
Section \ref{SEC:Hardy_Littlewood_Paley_inequalities}.
Using this, we get
\begin{equation}
\label{Trace}
\Tr\widehat{f}(k) = \int\limits_{\rSU2}f(u)\overline{\chi_k(u)}\,du,
\end{equation}
where $\chi_k(u)=\Tr T^k(u),\, k \in\frac12\mathbb{N}_0$, where we changed the notation
from $t^{k}$ to $T^{k}$ to avoid confusing with the notation that follows.
The characters $\chi_k(u)$ are constant on the conjugacy classes of $\rSU2$ and
we follow \cite{Vilenkin:BK-eng-1968} to describe these classes explicitly. 

It is well known from linear algebra that any unitary unimodular matrix $u$ can be written in the form $u=u_1 \delta u^{-1}_1$, where $u_1\in \rSU2$ and $\delta$ is a diagonal matrix of the form
\begin{equation}
\label{delta_matrix}
	\delta =
\begin{pmatrix}
e^{\frac{it}2} & 0 \\
0 & e^{-\frac{it}2}
\end{pmatrix},
\end{equation}
where $\lambda=e^{\frac{it}2}$ and $\frac1{\lambda}=e^{-\frac{it}2}$ are the eigenvalues of $u$. 
Moreover, among the matrices equivalent to $u$ there is only one other diagonal matrix, 
namely, the matrix $\delta'$ obtained from $\delta$ by interchanging the diagonal elements. 

Hence, classes of conjugate elements in $\rSU2$ are given by one parameter $t$, varying in the limits $-2\pi \leq t \leq 2\pi$, where the parameters $t$ and $-t$ give one and the same class.
Therefore, we can regard the characters $\chi_k(u)$ as functions of one variable $t$, which ranges from $0$ to $2\pi$.

The special unitary group $\rSU2$ is isomorphic to the group of unit quaternions. Hence, the parameter $t$ has a simple geometrical meaning - it is equal to angle of rotation which corresponds to the matrix $u$. 

Let us now derive an explicit  expression for the $\chi_k(u)$ as function of $t$. It was shown e.g. in \cite{RT} that  $T^k(\delta)$ is a diagonal matrix with the numbers $e^{-int},\,-k\leq n \leq k$ on its principal diagonal.

Let $u=u_1\delta u^{-1}_1$. Since characters are constant on conjugacy classes of elements, we get
\begin{equation}
\label{char_expl}
	\chi_k(u) = \chi_k(\delta) = \Tr {T^k(\delta)}=\sum_{n=-k}^{k}e^{int}.
\end{equation}
It is natural to express the invariant integral over $\rSU2$ in
\eqref{Trace}
in new parameters, one of which is $t$.

Since special unitary group $\rSU2$ is diffeomorphic to the unit sphere $\mathbb S^3$ in  
$\mathbb{R}^{4}$ (see, e.g., \cite{RT}), with
$$
\rSU2\ni u = 
\begin{pmatrix}
x_1+ix_2 & x_3+ix_4 \\
-x_3+ix_4 & x_1-ix_2
\end{pmatrix}
\longleftrightarrow
\varphi(u)=x=(x_1,x_2,x_3,x_4) \in \mathbb S^3,
$$
 we have 
\begin{equation}
\label{trace_aux_1}
	\int\limits_{\rSU2}f(u)\chi_k(u)\,du = \int\limits_{\mathbb S^3}f(x)\chi_k(x)\,dS,
\end{equation}
where $f(x):=f(\varphi^{-1}(x))$, and $\chi_k(x) := \chi_k(\varphi^{-1}(x))$.
In order to find an explicit formula for this integral over $\mathbb S^3$, 
we consider the parametrisation 
\begin{align*}
	&x_1 = \cos \frac{t}2,\\
	&x_2 =  v,\\
	&x_3 =\sqrt{\sin^2 \frac{t}2-v^2}\cdot\cos h,\\
	&x_4 =\sqrt{\sin^2 \frac{t}2-v^2}\cdot\sin h,\quad  (t,v,h)\in D,
\end{align*}
where $D=\{ (t,v,h)\in\mathbb{R}^3\colon |v|\leq \sin\frac{t}2, 0\leq t, h \leq 2\pi\}$.

The reader will have no difficulty in showing that
$
	dS = \sin \frac{t}2 dt dv dh.
$
Therefore, we have 
$$
   \int\limits_{\mathbb S^3}f(x)\chi_k(t) dS =
\int\limits_{D}f(h,v,t)\chi_k(t)\sin \frac{t}2 \,dhdvdt.
$$
Combining this and \eqref{trace_aux_1}, we get
$$
	\Tr \widehat{f}(k) = \int\limits_{D}f(h,v,t)\chi_k(t)\sin \frac{t}2 \,dhdvdt.
$$
Thus, we have expressed the invariant integral over $\rSU2$ in the parameters $t,v,h$.
The application of Fubini's Theorem yields
$$
 \int\limits_{D}f(h,v,t)\chi_k(t)\sin \frac{t}2\,dhdvdt  = 
\int\limits^{2\pi}_{0}\chi_k(t)\sin \frac{t}2\,dt\int\limits^{\sin \frac{t}2}_{-\sin \frac{t}2} dv \int\limits^{2\pi}_{0} f(h,v,t) \,dh .
$$
Combining this and \eqref{char_expl}, we obtain
$$
	\Tr \widehat{f}(k) = 
\int\limits^{2\pi}_{0}\,dt\sum_{n=-k}^{k}e^{int}\sin \frac{t}2\int\limits^{\sin \frac{t}2}_{-\sin \frac{t}2} dv \int\limits^{2\pi}_{0} f(h,v,t) \,dh .
$$
Interchanging summation and integration, we get
$$
	\Tr \widehat{f}(k) = \sum_{n=-k}^{k}
\int\limits^{2\pi}_{0} e^{int}\sin \frac{t}2\,dt\int\limits^{\sin \frac{t}2}_{-\sin \frac{t}2} dv \int\limits^{2\pi}_{0} f(h,v,t) \,dh .
$$
By making the change of variables $t\to 2t$, we get
\begin{equation}
\label{trace_explicit}
	\Tr \widehat{f}(k) = \sum_{n=-k}^{k}
\int\limits^{\pi}_{0} e^{-i2nt}\cdot2\sin t\,dt\int\limits^{\sin t}_{-\sin t} dv \int\limits^{2\pi}_{0} f(h,v,2t) \,dh .
\end{equation}

Let us now apply Theorem \ref{NED} in $L^p(\mathbb{T})$.  
To do this we introduce some notation.
Denote 
$$
F(t) := 
2\sin t\int\limits^{\sin t}_{-\sin t}  \int\limits^{2\pi}_{0} f(h,v,2t) \,dh\,dv, \quad t\in(0,\pi).
$$
We extend $F(t)$ periodically to $[0,2\pi)$, that is $F(x+\pi)=F(x)$. Since $f(t,v,h)$ is integrable, the integrability of $F(t)$ follows immediately from Fubini's Theorem. Thus function $F(t)$ has a Fourier series representation
$$
	F(t) \sim \sum_{k\in\mathbb{Z}}\widehat{F}(k)e^{ik t},
$$
where the Fourier coefficients are computed by
$$
	\widehat{F}(k)=\frac{1}{2\pi} \int\limits_{[0,2\pi]}F(t)e^{-ik t}\,dt.
$$
Let $A_k$ be a $2k+1$-element arithmetic sequence with step 2 and initial term $-2k$, i.e.,
$$
	A_k = \{-2k,-2k+2,\ldots, 2k\}=\{-2k+2j\}^{2k}_{j=0}.
$$
Using this notation and \eqref{trace_explicit}, we have
\begin{equation}\label{reduced}
	\Tr \widehat{f}(k)
=
\sum_{n\in A_k}\widehat{F}(n)	.
\end{equation}
Define
$$
	 B= \{A_k\}^{\infty}_{k=1}.
$$
Using the fact that	$B$ is a subset of the set $M$ of all finite arithmetic progressions,
and \eqref{reduced},
we have
\begin{equation}
\label{EQ:THM_necess_SU2_aux_1}
\sup_{\substack{k\in\frac12\mathbb{N}_0\\ 2k+1\geq 2l+1}}\frac1{2k+1}\left|\Tr \widehat{f}(k)\right|
\leq
\sup_{\substack{e \in B \\ |e|\geq 2l+1}}\frac1{|e|}\left|\sum_{i\in e}\widehat{F}(i)\right|
\leq
\sup_{\substack{e\in M \\ |e|\geq 2l+1}}\frac1{|e|}\left|\sum_{i\in e}\widehat{F}(i)\right|.
\end{equation}
Denote $m:=2l+1$. If $l$ runs over $\frac12\mathbb{N}_0$, then $m$ runs over $\mathbb{N}$. 
Using \eqref{EQ:THM_necess_SU2_aux_1}, we get

\begin{multline}\label{EQ:THM_necess_SU2_aux_2}
\sum_{l\in\frac12\mathbb{N}_0}(2l+1)^{p-2}\left(\sup_{\substack{k\in\frac12\mathbb{N}_0 \\ 2k+1\geq 2l+1}}\frac1{2k+1}\left|\Tr\widehat{f}(k)\right|\right)^p
\\ \leq
\sum_{m\in\mathbb{N}}m^{p-2}
\left(\sup_{\substack{e\in M \\ |e|\geq m}}\frac1{|e|}\left|\sum_{i\in e}\widehat{F}(i)\right|\right)^p.
\end{multline}
Application of  inequality \eqref{necess_T}   yields
\begin{equation}
\label{necess_S2U_aux_2}
\sum_{m\in\mathbb{N}}m^{p-2}
\left(\sup_{\substack{e\in M \\ |e|\geq m}}\frac1{|e|}\left|\sum_{i\in e}\widehat{F}(i)\right|\right)^p
\leq
c\|F\|^p_{L^p(0,2\pi)}.
\end{equation}
Using H\"older inequality, we obtain
$$
\int\limits^{\pi}_{0}|F(t)|^p\,dt
\lesssim
\int\limits^{\pi}_{0}\sin t\,dt\int\limits^{\sin t}_{-\sin t} dv \int\limits^{2\pi}_{0} |f(h,v,2t)|^p \,dh.
$$
By making the change of variables $t\to \frac{t}2$ in the right hand side integral, we get 
$$
\int\limits^{\pi}_{0}|F(t)|^p\,dt
\lesssim
\left(
\int\limits^{2\pi}_{0}\sin \frac{t}2\,dt\int\limits^{\sin \frac{t}2}_{-\sin \frac{t}2} dv \int\limits^{2\pi}_{0} |f(h,v,t)|^p \,dh
\right)^{\frac1p}.
$$
Thus, we have proved that
\begin{equation}
\label{EQ:THM_necess_SU2_aux_3}
\|F\|_{L^p(0,\pi)}
\leq
c_p\|f\|_{L^p(\rSU2)},
\end{equation}
where $c_p$ depending only on $p$.
Combining \eqref{EQ:THM_necess_SU2_aux_1}, \eqref{EQ:THM_necess_SU2_aux_2} and \eqref{EQ:THM_necess_SU2_aux_3}, we obtain
$$
		\sum_{m\in\mathbb{N}}m^{p-2}\left(\sup_{\substack{k\in\frac12\mathbb{N}_0 \\ 2k+1\geq m}}\frac1{2k+1}\left|\Tr\widehat{f}(k)\right|\right)^p
\leq
c\|f\|^{p}_{L^p(\rSU2)}.
$$
This completes the proof.
\end{proof}

\section{Marcinkiewicz interpolation theorem}
\label{SEC:Marc_Interpol_Theorem}
In this section we formulate and prove Marcinkiewicz interpolation theorem for linear 
mappings between $G$ and the space of matrix-valued sequences $\Sigma$ that will be realised via
$$\Sigma:=\left\{ h=\{h(\pi)\}_{\pi\in\Gh}, 
h(\pi)\in \mathbb{C}^{d_{\pi}\times d_{\pi}}\right\}.$$
Thus, a linear mapping $A\colon \mathcal D'(G) \to \Sigma$ takes a function to a matrix valued sequence, 
i.e.
$$
f \mapsto Af=:h=\{h(\pi)\}_{\pi\in\Gh},
$$
where
$$
h(\pi)\in \mathbb{C}^{d_{\pi}\times d_{\pi}},\; \pi\in\Gh.
$$
 We say that a linear operator $A$ is of strong type $(p,q)$, if for every $f\in L^p(G)$, 
 we have $Af\in \ell^q(\Gh,\Sigma)$ and
$$
	\|Af\|_{\ell^q(\Gh,\Sigma)}
\leq 
M
\|f\|_{L^p(G)},
$$
where $M$ is independent of $f$, and the space $\ell^q(\Gh,\Sigma)$ defined by the norm
\begin{equation}
\|h\|_{\ell^q(\Gh,\Sigma)}
:=
\left(
\sum\limits_{\pi\in\Gh}
d^{p(\frac2p-\frac12)}
\|h(\pi)\|^p_{HS}
\right)^{\frac1p}
\end{equation}
\eqref{EQ:lp}. The least $M$ for which this is satisfied is taken
to be the strong $(p,q)$-norm of the operator $A$.

Denote the distribution functions of $f$ and $h$ by $\mu_{G}(t;f)$ and $\nu_{\Gh}(u;h)$, 
respectively, i.e.
\begin{eqnarray}
	\mu_{G}(x;f)
:=
\int\limits_{\substack{u\in G \\ |f(u)|\geq x}}du,\, x>0,\\
\nu_{\Gh}(y;h)
:=
\sum\limits_{\substack{
\pi\in\Gh
\\ 
\frac{\|h(\pi)\|_{\HS}}{\sqrt{d_{\pi}}}
\geq y
}}d^2_{\pi},\,y>0.
\end{eqnarray}
Then
$$
\begin{aligned}
\|f\|^p_{L^p(G)}
&=
\int\limits_{G}
|f(u)|^p\,du
=
p\int\limits^{+\infty}_{0}x^{p-1}\mu_{G}(x;f)\,dx,
\\
\|h\|^q_{\ell^q(\Gh,\Sigma)}
&=
\sum\limits_{\pi\in\Gh}d^2_{\pi} \left(
\frac{\|h(\pi)\|_{\HS}}{\sqrt{d_{\pi}}}
\right)^q
=
q\int\limits^{+\infty}_{0}u^{q-1}\nu_{\Gh}(y;h)\,dy.
\end{aligned}
$$
A linear operator $A\colon \mathcal D'(\rSU2) \rightarrow \Sigma$ satisfying
\begin{equation}
\label{EQ:weak_type}
	\nu_{\Gh}(y;Af)
\leq
\left(
\frac{M}{y}\|f\|_{L^p(G)}
\right)^{q}
\end{equation}
is said to be of {\it weak type} $(p,q)$; the least value of $M$ in \eqref{EQ:weak_type} is called weak $(p,q)$ norm of $A$.

Every operation of strong type $(p,q)$ is also of weak type $(p,q)$, since
$$
	y\left(
	\nu_{\Gh}(y;Af)
	\right)^{\frac1q}
\leq
	\|Af\|_{L^q(\Gh)}
\leq
M
\|f\|_{L^p(G)}.
$$
\begin{thm} 
\label{THM:Marc-Interpol-Un-Dual}
Let $1\leq p_1<p<p_2<\infty$. Suppose that a linear operator $A$ 
from $\mathcal D'(G)$ to $\Sigma$
is simultaneously 
of {\it weak types} $(p_1,p_1)$ and $(p_2,p_2)$, with norms $M_1$ and $M_2$,
respectively, i.e.
\begin{eqnarray}
\label{EQ:weak_type_1}
	\nu_{\Gh}(y;Af)
\leq
\left(
\frac{M_1}{y}\|f\|_{L^{p_1}(G)}
\right)^{p_1},
\\
\label{EQ:weak_type_2}
	\nu_{\Gh}(y;Af)
\leq
\left(
\frac{M_2}{y}\|f\|_{L^{p_2}(G)}
\right)^{p_2}.
\end{eqnarray}
 Then for any $p\in(p_1,p_2)$ the operator $A$ is of strong type $(p,p)$ and we have
\begin{equation}
\label{EQ:norm_interpolation_estimate}
\|Af\|_{\ell^p(\Gh,\Sigma)}
\leq
M^{1-\theta}_1M^{\theta}_2\|f\|_{L^p(G)},\quad 0<\theta<1,
\end{equation}
where
\begin{eqnarray*}
\frac1p=\frac{1-\theta}{p_1}+\frac{\theta}{p_2}.
\end{eqnarray*}
\end{thm}
The proof is done in analogy to Zygmund \cite{Zygmund:Marc-JMPA-1956} adapting it
to our setting.
\begin{proof}
Let $f\in L^p(G)$. We have to prove inequality \eqref{EQ:norm_interpolation_estimate}. By definition, we have
\begin{equation}
\label{EQ:THM_Marc_Interpol_Un_Dual_aux_1}
	\|Af\|_{\ell^p(\Gh,\Sigma)}^p
=
	\sum\limits_{\pi\in\Gh}d^2_{\pi}
	\left(
	\frac{\|Af(\pi)\|_{\HS}}{\sqrt{d_{\pi}}}
	\right)^p
	=
\int\limits^{+\infty}_0
px^{p-1}
\nu_{\Gh}(x;Af)
\,dx.
\end{equation}

 For a fixed arbitrary $z>0$ we consider the decomposition
$$
f=f_1+f_2,
$$
where $f_1=f$ whenever $|f|<z$, and $f_1=0$ otherwise; thus $|f_2|>z$ or else $f_2=0$. Since $f\in L^{p}(G)$ the same holds for $f_1$ and $f_2$; it follows that $f_1$ is in $L^{p_1}(G)$ and $f_2\in L^{p_2}(G)$. Hence $Af_1$ and $Af_2$ exist, by hypothesis, and so does $Af=A(f_1+f_2)$. It follows that
\begin{equation}
\label{EQ:decomposition}
	|f_1|=\min(|f|,z),\quad |f|=|f_1|+|f_2|.
\end{equation}
The inequality
$$
	\|A(f_1+f_2)(\pi)\|_{\HS}
\leq
\|Af_1(\pi)\|_{\HS}+
\|Af_2(\pi)\|_{\HS},\,\pi\in\Gh
$$
leads to an inclusion
\begin{multline*}
	\left
	\{
	\pi\in\Gh
	\colon
	\frac{\|Af(\pi)\|_{\HS}}{\sqrt{d_{\pi}}}
	\geq y
	\right
	\}
\subset \\ \subset
	\left
	\{
	\pi\in\Gh
	\colon
	\frac{\|Af_1(\pi)\|_{\HS}}{\sqrt{d_{\pi}}}
	\geq \frac{y}2
	\right
	\}
\bigcup
	\left
	\{
	\pi\in\Gh
	\colon
	\frac{\|Af_2(\pi)\|_{\HS}}{\sqrt{d_{\pi}}}
	\geq \frac{y}2
	\right
	\}.
\end{multline*}
Then applying assumptions \eqref{EQ:weak_type_1} and \eqref{EQ:weak_type_2} to 
$f_1$ and $f_2$,  we obtain
\begin{multline}
\label{EQ:THM_Marc_Interpol_Un_Dual_aux_2}
	\nu_{\Gh}(y;Af)
\leq
	\nu_{\Gh}(\frac{y}2;Af_1)
+
	\nu_{\Gh}(\frac{y}2;Af_2)
\\\leq
	M^{p_1}_1y^{-p_1}\|f_1\|^{p_1}_{L^{p_1}(G)}
+
	M^{p_2}_1y^{-p_2}\|f_2\|^{p_2}_{L^{p_2}(G)}.
\end{multline}
The right side here depends on $z$ and the main idea of the proof consists in defining $z$ as a suitable monotone functions of $t, z=z(t)$, to be determined later.
By \eqref{EQ:decomposition}
\begin{eqnarray*}
\mu_{G}(t;f_1) &= &\mu_{G}(t;f),				\quad \text{ for } 0<t\leq z,
\\
\mu_{G}(t;f_1)&=&0,		     \qquad \qquad		\qquad \text{ for } t>z,
\\
\mu_{G}(t;f_2)&=&\mu_{G}(t+z;f),			\quad \text{ for } t>0.
\end{eqnarray*}
Here, the last equation is a consequence of the fact that wherever $f_2\neq 0$ we must have $|f_1|=z$, and so the second equation \eqref{EQ:decomposition} takes the form $|f|=z+|f_2|$.

It follows from \eqref{EQ:THM_Marc_Interpol_Un_Dual_aux_2} that the last integral in \eqref{EQ:THM_Marc_Interpol_Un_Dual_aux_1} is less than
\begin{multline}
\label{EQ:THM_Marc_Interpol_Un_Dual_aux_3}
M^{p_1}_1\int\limits^{+\infty}_{0}y^{p-p_1-1}
\left
\{
\int\limits_{G}|f_1(u)|^{p_1}\,du
\right
\}^{\frac{p_1}{p_1}}\,dy+
\\+
M^{p_2}_2\int\limits^{+\infty}_{0}y^{p-p_2-1}
\left
\{
\int\limits_{G}|f_2(u)|^{p_1}\,du
\right
\}^{\frac{p_2}{p_2}}\,dy
\\=
M^{p_1}_1
p_1
\int\limits^{+\infty}_{0}y^{p-p_1-1}
\left
\{
\int\limits^z_{0}x^{p_1-1}\mu_{G}(x;f)\,dx
\right
\}\,dt
\\ +
M^{p_2}_2
p_2
\int\limits^{+\infty}_{0}y^{p-p_2-1}
\left
\{
\int\limits^{+\infty}_{z}(x-z)^{p_2-1}\mu_{G}(x;f)\,dx
\right
\}\,dt.
\end{multline}
Set $z(y)=\frac{A}{y}$.
Denote by $I_1$ and $I_2$ the two double integrals last written. 
We change the order of integration in $I_1$
\begin{multline}
\label{EQ:THM_Marc_Interpol_Un_Dual_aux_4}
I_1=
\int\limits^{+\infty}_0
t^{p-p_1-1}
\left
\{
\int\limits^z_0
u^{p_1-1}
\mu_{G}(u;f)
\,du
\right
\}
\,dt	
\\=
\int\limits^{+\infty}_0
x^{p_1-1}
\mu_{G}(x;f)
\left
\{
\int\limits^{Ax}_{0}
y^{p-p_1-1}
\,dy	
\right
\}
\,dx
\\=
\frac{A^{p-p_1}}{p-p_1}
\int\limits^{+\infty}_0
x^{p_1-1+p-p_1}
\mu_{G}(x;f)
\,dx
.
\end{multline}
Similarly, making a substitution $x-z \rightarrow x$ and using \eqref{EQ:decomposition} we see that $I_2$ is
\begin{multline}
\label{EQ:THM_Marc_Interpol_Un_Dual_aux_5}
I_2
=
M^{p_2}_2
p_2
\int\limits^{+\infty}_{0}y^{p-p_2-1}
\left
\{
\int\limits^{+\infty}_{z}(x-z)^{p_2-1}\mu_{G}(x;f)\,dx
\right
\}\,dy
\\=
M^{p_2}_2
p_2
\int\limits^{+\infty}_{0}y^{p-p_2-1}
\left
\{
\int\limits^{+\infty}_{0}x^{p_2-1}\mu_{G}(x+z;f)\,dx
\right
\}\,dy
\\=
M^{p_2}_2
p_2
\int\limits^{+\infty}_{0}y^{p-p_2-1}
\left
\{
\int\limits^{+\infty}_{0}x^{p_2-1}\mu_{G}(x;f_2)\,dx
\right
\}\,dy
\\=
M^{p_2}_2
p_2
\int\limits^{+\infty}_{0}
\left
\{
\int\limits^{+\infty}_{0}x^{p_2-1}
\mu_{G}(x;f_2)y^{p-p_2-1}\,dy
\right
\}\,dx
\\=
M^{p_2}_2
p_2
\int\limits^{+\infty}_{0}
\left
\{
\int\limits^{+\infty}_{Ax^{\frac1{\xi}}}x^{p_2-1}
\mu_{G}(x;f_2)y^{p-p_2-1}\,dy
\right
\}\,dx
\\=
M^{p_2}_2
p_2
\int\limits^{+\infty}_{0}
x^{p_2-1}
\mu_{G}(x;f_2)
\left
\{
\int\limits^{+\infty}_{Ax}y^{p-p_2-1}\,dy
\right
\}\,dx
\\=
\frac{A^{p-p_2}}{p_2-p}
M^{p_2}_2
p_2
\int\limits^{+\infty}_{0}
x^{p_2-1+p-p_2}
\mu_{G}(x;f_2)
\,dx
\\\leq
\frac{A^{p-p_2}}{p_2-p}
M^{p_2}_2
p_2
\int\limits^{+\infty}_{0}
x^{p_2-1+p-p_2}
\mu_{G}(x;f)
\,dx
.
\end{multline}
Collecting estimates
\eqref{EQ:THM_Marc_Interpol_Un_Dual_aux_3}, 
\eqref{EQ:THM_Marc_Interpol_Un_Dual_aux_4},
\eqref{EQ:THM_Marc_Interpol_Un_Dual_aux_5} we see that integral in \eqref{EQ:THM_Marc_Interpol_Un_Dual_aux_1} does not exceed
\begin{multline}
\label{EQ:THM_Marc_Interpol_Un_Dual_aux_6}
M^{p_1}_1
p_1
\frac{A^{p-p_1}}{p-p_1}
\int\limits^{+\infty}_0
x^{p-1}
\mu_{G}(x;f)
\,dx
+
M^{p_2}_2
p_2
\frac{A^{p-p_2}}{p_2-p}
\int\limits^{+\infty}_{0}
x^{p-1}
\mu_{G}(x;f_2) dx
.	
\end{multline}
Now, using the identity
$$
\int\limits^{+\infty}_0
x^{p-1}
\mu_{G}(x;f)
\,dx
=
\int\limits_{G}	
|f(u)|^{p}\,du
=
\|f\|^{p}_{L^p(G)},
$$
and inequalities \eqref{EQ:THM_Marc_Interpol_Un_Dual_aux_1} and \eqref{EQ:THM_Marc_Interpol_Un_Dual_aux_6}
we get
$$
	\|Af\|^p_{\ell^p(\Gh)}
	\leq
	\left(
	M^{p_1}_1
	p_1
	\frac{A^{p-p_1}}{p-p_1}
	+
	M^{p_2}_2
	p_2
	\frac{A^{p-p_2}}{p_2-p}
	\right)^p
	\|f\|^p_{\ell^p(\Gh)}.
$$
Next we set
$$
	A=M^{\frac{p_1}{p_1-p_2}}_1 M_{2}^{\frac{p_2}{p_2-p_1}}.
$$
A simple computation shows that
$$
M^{p_1}_1A^{p-p_1}
=
M^{p_2}_2
A^{p-p_2}
=
M^{\frac{p_1(p_2-p)}{p_2-p_1}}_1
M^{\frac{p_2(p_1-p)}{p_1-p_2}}_2
=
M^{1-\theta}_1 M^{\theta}_2,\quad \frac1p=\frac{1-\theta}{p_1}+\frac{\theta}{p_2}.
$$
Finally, we have
$$
\|Af\|_{\ell^p(\Gh)}
\leq
K_{p,p_1,p_2}
M^{1-\theta}_1 M^{\theta}_2
\|f\|_{L^p(G)},
$$
where
$$
	K_{p,p_1,p_2}
	=
\left(
\frac{p_1}{p-p_1}+
\frac{p_2}{p_2-p}
\right)^{\frac1p}.
$$
This completes the proof.
\end{proof}

\end{document}